\documentclass[12pt]{amsart}
\usepackage{amssymb,amscd}
\usepackage{verbatim}

\usepackage{tikz}

\let\frak\mathfrak
\let\Bbb\mathbb

\def\>{\relax\ifmmode\mskip.666667\thinmuskip\relax\else\kern.111111em\fi}
\def\<{\relax\ifmmode\mskip-.333333\thinmuskip\relax\else\kern-.0555556em\fi}
\def\vsk#1>{\vskip#1\baselineskip}
\def\vv#1>{\vadjust{\vsk#1>}\ignorespaces}
\def\vvn#1>{\vadjust{\nobreak\vsk#1>\nobreak}\ignorespaces}

  \let\ssize\scriptstyle
\let\sssize\scriptscriptstyle

\let\Medskip\medskip
\def\medskip{\par\Medskip}
\let\Bigskip\bigskip
\def\bigskip{\par\Bigskip}

\let\Maketitle\maketitle
\def\maketitle{\Maketitle\thispagestyle{empty}\let\maketitle\empty}

\newtheorem{thm}{Theorem}[section]
\newtheorem{cor}[thm]{Corollary}
\newtheorem{lem}[thm]{Lemma}

\numberwithin{equation}{section}

\theoremstyle{definition}
\newtheorem*{rem}{Remark}
\newtheorem*{example}{Example}

\let\mc\mathcal
\let\nc\newcommand

\let\al\alpha

\let\ka\kappa
\let\la\lambda

\let\phi\varphi

\let\Si\Sigma

\let\om\omega

\let\der\partial

\let\geq\geqslant

\let\leq\leqslant

\let\on\operatorname
\let\bi\bibitem
\let\bs\boldsymbol

\def\C{{\mathbb C}}
\def\Z{{\mathbb Z}}

\def\F{{\mc F}}

\def\+#1{^{\{#1\}}}

\def\End{\on{End}}

\def\Res{{\on{Res}}}

\def\sln{\mathfrak{sl}_N}

\def\beq{\begin{equation}}
\def\eeq{\end{equation}}
\def\be{\begin{equation*}}
\def\ee{\end{equation*}}

\nc{\bea}{\begin{eqnarray*}}
\nc{\eea}{\end{eqnarray*}}
\nc{\bean}{\begin{eqnarray}}
\nc{\eean}{\end{eqnarray}}
\nc{\Ref}[1]{{\rm(\ref{#1})}}

\let\ga\gamma

\nc{\Il}{{\mc I_{\bs\la}}}
\nc{\bla}{{\bs\la}}
\nc{\Fla}{\F_\bla}
\nc{\tfl}{{T^*\Fla}}
\nc{\GL}{{GL_n(\C)}}
\nc{\GLC}{{GL_n(\C)\times\C^*}}

\let\sd s 

\def\ddk_#1{\kk_{#1}\<\>\frac\der{\der\<\>\kk_{#1}}}

\def\bul{\mathbin{\raise.2ex\hbox{$\sssize\bullet$}}}
\def\intt{\mathchoice
{\mathop{\raise.2ex\rlap{$\,\,\ssize\backslash$}{\intop}}\nolimits}
{\mathop{\raise.3ex\rlap{$\,\sssize\backslash$}{\intop}}\nolimits}
{\mathop{\raise.1ex\rlap{$\sssize\>\backslash$}{\intop}}\nolimits}
{\mathop{\rlap{$\sssize\<\>\backslash$}{\intop}}\nolimits}}

\let\kk q 
\let\cc c

\let\Ko K

\def\GZ/{Gelfand-Zetlin}
\def\KZ/{{\slshape KZ\/}}
\def\qKZ/{{\slshape qKZ\/}}
\def\XXX/{{\slshape XXX\/}}

\nc{\slnl}{{\sln (\lambda)}}
\nc{\PCN}{{   (\C[x])^N   }}
\nc{\di}{\on{Diag}}
\nc{\dio}{\on{Diag}_0}
\nc{\Mm}{{\mc M}}
\nc{\Nn}{{\mc N}}

\nc{\A}{{\mc C}}

\nc{\PCr}{{  P  (\C[x])^n   }}

\nc{\Pk}{{(\bs{P}^1)^k}}

\nc{\N}{{\Bbb N}}

\nc{\Ll}{{\mc L}}

\nc{\ord}{{\on{ord}\,}}

\newcommand{\OS}{\mathcal {A}}
\def\FF{{\mathcal F}}

\nc{\Sing}{{\on{Sing}\,}}
\nc{\sing}{{\on{Sing}\,}}

\nc{\Hess}{{\on{Hess}}}

\nc{\R}{{\Bbb R}}
\newcommand{\Pee}{{\mathbb P}}
\let\on\operatorname
\nc{\Kk}{{\bs K}}
\nc{\Ap}{{A_\Phi(z)}}
\nc{\ap}{{A_\Phi(z)}}

\nc{\sv}{{\on{Sing}_a V}}
\nc{\cd}{{\C^n-\Delta}}
\nc{\UT}{{U^0}}   
\nc{\Spect}{\on{Spec}\nolimits}

\nc\z{{\bs z}}
\nc\p{{\bs p}}
\nc\q{{\bs q}}
\nc\ttt{{\bs t}}
\nc\OC{{\mc O(C_{\A,a})}}
\nc\Ia{{\mc I}}
\nc\OCx {{\mc O(C_{\A(x),a})}}
\nc\Cs{{(\C^k)^*}}
\nc\OLx{{ \mc O(L_{Y,a}(x))}}

\begin{document}

\hrule width0pt
\vsk->

\title[Potentials of a family of arrangements of hyperplanes ]
{Potentials of a family of arrangements of hyperplanes and elementary subarrangements }

\author
[Andrew Prudhom and Alexander Varchenko]
{Andrew Prudhom and  Alexander  Varchenko$\>^\diamond$}

\maketitle

\begin{center}

\vsk.5>
{\it Department of Mathematics, University
of North Carolina at Chapel Hill\\ Chapel Hill, NC 27599-3250, USA\/}

\end{center}

{\let\thefootnote\relax
\footnotetext{\vsk-.8>\noindent
AMS Subject Classification: 32S22, 53D45, 14N20
\\
$^\diamond\<${\sl E\>-mail}: \enspace anv@email.unc.edu\>,
supported in part by  NSF grant DMS-1362924
          and Simons Foundation grant \#336826.}}

\medskip

\begin{abstract}
We consider  the Frobenius algebra of functions on the critical set of the master function of a weighted arrangement of hyperplanes in $\C^k$ with normal crossings.  We construct two {\it potential} functions 
(of first and second kind) of variables labeled by hyperplanes of the arrangement and prove that the matrix coefficients of the Grothendieck residue bilinear form on the algebra are given by the $2k$-th derivatives of the potential function of first kind and the matrix coefficients of the multiplication operators on the algebra are given by the $(2k+1)$-st derivatives of the potential function of second kind.
Thus the two potentials completely determine the Frobenius algebra. The presence of these potentials is a manifestation of a Frobenius like structure similar to the Frobenius manifold structure.

We introduce the notion of an elementary subarrangement of an arrangement with normal crossings. 
It turns out that our potential functions are local in the sense that the potential functions are sums of contributions from elementary subarrangements of the given arrangement. This is a new phenomenon of  locality of the Grothendieck residue bilinear form and multiplication on the algebra.

It is known that this Frobenius algebra of functions on the critical set is isomorphic
to the Bethe algebra of this arrangement. (That Bethe algebra is an analog of the Bethe algebras in the theory of quantum integrable models.) Thus our potential functions describe that Bethe algebra too. 

\end{abstract}


\setcounter{footnote}{0}
\renewcommand{\thefootnote}{\arabic{footnote}}

\section{Introduction}
It is well known that the algebra of functions on the set of solutions of the Bethe ansatz equations plays an important role in the study of quantum integrable systems since in many cases the algebra of functions is isomorphic to the Bethe algebra of Hamiltonians of the system,
see for example \cite{NS, MTV1, GRTV, R}.  
 An interesting problem is to describe the algebra.
In this paper we consider the model case of the algebra of functions on the critical set of the master function associated with a family of arrangements with normal crossings. 
Such algebras appear in the KZ-Gaudin type integrable systems, see for example \cite{SV,RV}.
We describe the algebra of functions on the critical set together with the Grothendieck residue
bilinear form in terms of derivatives of two potential functions in the spirit of Frobenius structures.

\subsection{Statement of results}
\label{Sr}
Denote $J=\{1,\dots,n\}$. 
Consider $\C^n\times\C^k$ with coordinates $(z,t)=(z_1,\dots,z_n,t_1,\dots,t_k)$ and  the projection
$\tau :\C^n\times\C^k \to \C^n$.
Fix $n$ nonzero linear functions on $\C^k$,
$g_j=b_j^1t_1+\dots + b_j^kt_k,$\ $ j\in J$.  Assume that  $\{g_j\}_{ j\in J}$
span the dual space $(\C^k)^*$.
Define the functions $f_j = g_j + z_j$ on $\C^n\times\C^k$. We obtain on $\C^n\times\C^k$ an arrangement $\A=\{H_j\}_{j=1}^n$, where $H_j$ is the zero set of $f_j$.  
Let $U(\A): =\C^n\times \C^k-\cup_{j\in J}H_j$ be the complement.
For every $x\in\C^n$, the arrangement $\A$ restricts to an 
arrangement $\A(x)$ on $\tau^{-1}(x)\cong\C^k$
with the complement $U(\A(x)):=\tau^{-1}(x) \cap U(\A)$.
For almost all $x\in\C^k$ the arrangement $\A(x)$ is 
with normal crossings. The subset $\Delta\subset\C^n$, where
this does not hold, is a hypersurface and is called the discriminant.

A set $I=\{i_1,...,i_k\}\subset J$ is called independent if
 $g_{i_1},...,g_{i_k}$ are linearly independent. Denote
$J^{\on{ind}}$ the set of all independent $k$-element subsets of $J$.

Let $a=(a_1,...,a_n)\in(\C^*)^n$ be a system of { weights}
such that for any $x\in \C^n-\Delta$ the weighted arrangement
$(\A(x),a)$ is {unbalanced}, see Section \ref{IsCrPt}, e.g. $a\in\R_{>0}^n$ is unbalanced,
also a generic system of weights is     unbalanced.
The {\it master function} of the weighted arrangement $(\A,a)$
is
\begin{eqnarray}
\label{4.12}
\Phi_{\A,a}(z,t):={\sum}_{i\in J}a_i\log f_i.
\end{eqnarray}
For $x\in\C^n-\Delta$ all critical points of $\Phi_{\A, a}|_{z=x}$ 
with respect to the variables $t$, are isolated,
and the sum $\mu$ of their Milnor numbers is independent of 
the unbalanced weight $a$ and the parameter $x\in\C^n-\Delta$. 
The main object of this paper is the $\mu$-dimensional algebra
\bean
\label{mo}
\mc O(C_{\A(x),a}): = \mc O(U(\A(x)))\Big/\Big(\frac{\der\Phi_{\A,a}}{\der t_j}\, \Big|\,
j=1,...,k\Big)
\eean
of functions on the critical set of the master function $\Phi_{\A,a}\vert_{z=x}$, 
see Section \ref{Cr set}. Define 
\bean
\label{def p}
p_j:=\Big[\frac{a_j}{f_j}\Big]= \Big[\frac{\der \Phi_{\A,a}}{\der z_j}\Big] \in \mc O(C_{\A(x),a}),
\qquad j\in J.
\eean
The elements $\{p_j\}_{ j\in J}$ generate $\mc O(C_{\A(x),a})$ as an algebra. 
The elements  $\{p_{i_1}\cdots p_{i_k}\}_{\{i_1,\dots,i_k\}\in J^{\on{ind}}}$ generate
$\mc O(C_{\A(x),a})$ as a vector space. 
 The Grothendieck residue 
defines a nondegenerate  bilinear form $(\,,\,)_{C_{\A(x),a}}$ on 
$\mc O(C_{\A(x),a})$. 
 The algebra $(\mc O(C_{\A(x),a}),(\,,\,)_{C_{\A(x),a}})$
is a Frobenius algebra. 

The main result of this paper is a construction of
two functions $P$, $Q$ on $\C^n-\Delta$ called the 
{\it potentials of first and second kind}, respectively. The potentials  have the following properties.
 
\begin{thm}
\label{introT}
 Let $x\in\C^n-\Delta$. Then for 
 any two  independent subsets $\{i_1,\dots,i_k\}, \{l_1,\dots,l_k\}$ $ \subset J$ and any $i_0\in J$, we have
\bean
\label{Q11F}
(p_{i_1}\cdots p_{i_k}, p_{l_1}\cdots p_{l_k})_{C_{\A(x),a}}
=
(-1)^k
 \frac{\der^{2k}P}{\der z_{i_1}\dots\der z_{i_k}\der z_{l_1}\dots\der z_{l_k}}(x),
\eean
\bean
\label{Q12F}
(p_{i_0}p_{i_1}\cdots p_{i_k}, p_{l_1}\cdots p_{l_k})_{C_{\A(x),a}}
=
(-1)^k
 \frac{\der^{2k+1}Q}{\der z_{i_0}\der z_{i_1}\dots\der z_{i_k}\der z_{l_1}\dots\der z_{l_k}}(x).
\eean
\end{thm}

Formula \Ref{Q11F} determines the Grothendieck residue bilinear form $(\,,\,)_{C_{\A(x),a}}$ in terms of the potential of  first kind. Formula
\Ref{Q12F} determines the operators of multiplication by generators $\{p_j\}_{j\in J}$ in terms of the potential of second kind.

\begin{example} For the arrangement of four lines shown in Figure 1 and given by equations
$t_2+        z_1=0$,
$  t_2+        z_2=0$,
$   t_1+z_3=0$,  $  t_1+t_2+z_4=0$ we have
\begin{figure}
  \begin{tikzpicture}
 
\draw[thick] (.5,2) -- (.5,0) node[anchor=west,yshift=5pt] {$H_{3}$};
\draw[thick] (0,2) -- (2,0) node[anchor=west,yshift=5pt] {$H_{4}$};
\draw[thick] (-.5,0.75) -- (2.5,0.75) node[anchor=west] {$H_{2}$ };
\draw[thick] (-.5,1.25) -- (2.5,1.25) node[anchor=west] {$H_{1}$};

\end{tikzpicture}
\caption{}
\end{figure}
\bea
&&
P = \frac1{a_1+a_2+a_3+a_4}
\Big(a_1a_3a_4 \frac{(z_1+z_3-z_4)^4}{4!}
+ a_2a_3a_4 \frac{(z_2+z_3-z_4)^4}{4!}
\\
\notag
&&
\phantom{aaaaaaaaaaaaaaaaaaa}
+ \frac{a_1a_2a_3a_4}{a_3+a_4} 
 \frac {(z_1-z_2)^2}{2!}\frac{(z_1+z_3-z_4)^2}{2!} \Big),
\eea
\bea
&&
Q =
a_1a_3a_4 \ln(z_1+z_3-z_4)\frac{(z_1+z_3-z_4)^4}{4!}
+ a_2a_3a_4\ln(z_2+z_3-z_4) \frac{(z_2+z_3-z_4)^4}{4!}
\\
\notag
&&
\phantom{aaaaaaaaaaaaaa}
+ \frac{a_1a_2a_3a_4}{a_3+a_4} 
 \ln(z_1-z_2)\frac {(z_1-z_2)^2}{2!}\frac{(z_1+z_3-z_4)^2}{2!} .
\eea
Theorem \ref{introT}  in particular says that $(p_1p_3, p_2p_4)_{C_{\A(x),a}} = \frac 
{a_1a_2a_3a_4}
{(a_1+a_2+a_3+a_4)(a_3+a_4)}$ and it does not depend on $x\in\C^n-\Delta$, and
$(p_4p_1p_3,p_3p_4)_{C_{\A(x),a}}$ 
$= 
\frac{a_1a_3a_4}{z_1+z_3-z_4}$.
\end{example}

In this example the potentials are sums of terms corresponding to subarrangements 
consisting of three or four lines. It turns out that this is the general case. In  Section \ref{sdef}
we introduce the notion of an elementary arrangement in $\C^k$ of type $\la=(\la_1,\dots,\la_m)$, $\la_h\in\Z_{>0}$, $\la_1+\dots+\la_m=k$.  In particular, such an elementary arrangement consists of 
$k+m$ hyperplanes, and  an elementary arrangement in $\C^k$
has at most $2k$ hyperplanes. 
We show that the potentials are sums, over all elementary subarrangements, of the prepotentials of the subarrangements taken with suitable weights,  see Corollary \ref{corP1} and Theorem   \ref{thmQ1}.
The fact that the potentials are sums of contributions from elementary subarrangements indicates a new phenomenon of {\it locality} of the Grothendieck residue bilinear form and multiplication on $\mc O(C_{\A(x),a})$.

The existence of the potentials of first and second kind locally on $\C^n-\Delta$ 
was established in \cite{HV}.

\subsection{Frobenius like structure of order $(n,k,m)$ }  
\label{s1.2}
The potential of the second kind is an analog of the potential in the theory of Frobenius manifolds.
A Frobenius manifold is a manifold with a flat metric and a Frobenius algebra structure on  tangent spaces 
at points of the manifold such that the structure constants of multiplication are given by third derivatives
of a potential function on the manifold with respect to  flat coordinates, see \cite{D, M}.
As an analogy of that, for  our family of arrangements the structure constants of multiplication are given by $2k+1$-st derivatives of the potential of second kind, see Theorem \ref{introT}.

The notion of potentials of a family of arrangements was introduced and studied in \cite{V6, V9, HV}. In \cite{V6} the potentials were constructed for the families of {\it generic arrangements}, that is,  
such that the linear functions $g_{i_1},\dots,g_{i_k}$ are linearly independent for any distinct $i_1,\dots,i_k\in J$.
In  \cite{V6, V9, HV}  different axiomatizations of the structure leading to the existence of the potentials 
 were given. In particular in \cite{HV} Frobenius like structures of order $(n,k,m)$ were introduced.
Our case of a family of arrangements corresponds to the case of order $(n,k,2)$.   Under the axioms of \cite{HV} the existence of
the potential of second kind was deduced in \cite{HV} from
a surprising  elementary study of finite sets of vectors    in a finite-dimensional vector space $W$. Given a natural number $m$  and a finite set  $\{w_i\}$ of vectors, a necessary and sufficient condition was  given
to find in the set $\{w_i\}$ $m$ bases of  $W$.  If $m$ bases in the set $\{w_i\}$  are selected,  then  some elementary transformations of such  a selection are defined. It was shown in \cite{HV} that any two selections are connected by a sequence of elementary transformations. These structures are  fundamental and one may expect a matroid version of them.

\subsection{Bethe algebra} Given a family of weighted arrangements in  $\C^k$ as in Section \ref{Sr}, one considers the Gauss-Manin differential equations for associated $k$-dimensional hypergeometric integrals,
$\ka\frac{\der I}{\der z_j}(z)= K_j(z)I(z)$, $j\in J$, $z\in\C^n-\Delta$,  where $K_i(z)$ are suitable linear operators on the space of singular vectors $\Sing_aV$,  see Section \ref{Constrn}. For every $x\in\C^n-\Delta$, the operators $K_j(x), j\in J$, commmute and are symmetric with respect to the contravariant bilinear  form $S^{(a)}$
on  $\Sing_aV$.  The   unital subalgebra of $\End(\Sing_aV)$ generated by the operators
 $K_j(x), j\in J$,  is called the Bethe algebra of the weighted arrangement $(\A(x),a)$. This algebra is the analog of the Bethe algebra in the theory of quantum integrable systems, see \cite{V5}.
It is known that the Bethe algebra together with  the bilinear form
$S^{(a)}$ is isomorphic to the pair consisting of the algebra of multiplication operators on $\mc O(C_{\A(x),a})$ and
the Grothendieck residue bilinear form $(\,,\,)_{C_{\A(x),a}}$.  Thus Theorem \ref {introT} gives us a description of the Bethe algebra in terms of the derivatives of the potential functions, see Theorem \ref{thmP1}  and Corollary \ref{corQ}.

Our construction of potential functions is based on the isomorphism of the Bethe algebra and the algebra of functions on the critical set.

The Bethe algebra of our family of arrangements is a toy example of the Bethe algebra 
of a quantum integrable system. One may expect to determine glimpses of Frobenius 
like structures in the Bethe algebras of standard quantum integrable systems. 

\subsection{Exposition of material} In Section \ref{ARR} we remind general facts about arrangements. In Section \ref{secT Tr}
we consider families of arrangements. In Section \ref{EA} we introduce elementary arrangements and define  
potential functions. In Section \ref{OPR} we prove an important formula for the orthogonal projection $\pi : V\to \Sing_aV$
with respect to the bilinear form $S^{(a)}$.  Based on that formula we prove the first part of Theorem \ref{introT} in Section \ref{PT1} and the second part of Theorem \ref{introT} in Section \ref{PT2}.

\medskip
The second author thanks the MPI in Bonn for hospitality during his visit in 2015-2016,
C.\,Hertling and B.\,Dubrovin for useful discussions. We also thank C.\,Hertling for indicating a mistake in the initial draft of the paper.

\section{Arrangements}
\label{ARR}

\subsection{Affine arrangement}
\label{AFF}

Let $k,n$ be positive integers, $k<n$. Denote $J=\{1,\dots,n\}$.

Consider the complex affine space $\C^k$ with coordinates $t_1,\dots,t_k$.
Let $\A =(H_j)_{j\in J}$,  be an arrangement of $n$ affine hyperplanes in
$\C^k$. Denote $U(\A) = \C^k - \cup_{j\in J} H_j$,
the complement.
An {\it edge} $X_\al \subset \C^k$ of $\A$ is a nonempty intersection of some
hyperplanes  of $\A$. Denote by
 $J_\al \subset J$ the subset of indices of all hyperplanes containing $X_\al$.
Denote $l_\al = \mathrm{codim}_{\C^k} X_\al$.
We assume that $\A$ is {\it essential}, that is, $\A$ has a vertex, an edge which is a point.

An edge is called {\it dense} if the subarrangement of all hyperplanes containing
it is irreducible: the hyperplanes cannot be partitioned into nonempty
sets so that, after a change of coordinates, hyperplanes in different
sets are in different coordinates. In particular, each hyperplane of
$\A$ is a dense edge.

\subsection{Orlik-Solomon algebra}
Define complex vector spaces $\OS^p(\A)$, $p = 0,  \dots, k$.
 For $p=0$, we set $\OS^p(\A)=\C$. For  $p \geq 1$,\
 $\OS^p(\A)$   is generated by symbols
$(H_{j_1},...,H_{j_p})$ with ${j_i}\in J$, such that
\ (i) $(H_{j_1},...,H_{j_p})=0$
if $H_{j_1}$,...,$H_{j_p}$ are not in general position, that is, if the
intersection $H_{j_1}\cap ... \cap H_{j_p}$ is empty or
 has codimension
 less than $p$;\
(ii)
$ (H_{j_{\sigma(1)}},...,H_{j_{\sigma(p)}})=(-1)^{|\sigma|}
(H_{j_1},...,H_{j_p})
$
for any element $\sigma$ of the
symmetric group $ \Sigma_p$;
\ (iii)
$\sum_{i=1}^{p+1}(-1)^i (H_{j_1},...,\widehat{H}_{j_i},...,H_{j_{p+1}}) = 0
$
for any $(p+1)$-tuple $H_{j_1},...,H_{j_{p+1}}$ of hyperplanes
in $\A$ which are
not in general position and such that $H_{j_1}\cap...\cap H_{j_{p+1}}\not = \emptyset$.

The direct sum $\OS(\A) = \oplus_{p=1}^{N}\OS^p(\A)$ is the {\it Orlik-Solomon}
 algebra with respect to  multiplication
 $ (H_{j_1},...,H_{j_p})\cdot(H_{j_{p+1}},...,H_{j_{p+q}}) =
 (H_{j_1},...,H_{j_p},H_{j_{p+1}},...,H_{j_{p+q}})$.

\subsection{Aomoto complex }
\label{sec wts}
 Fix a point $a=(a_1,\dots,a_n)\in (\C^\times)^n$ called the {\it weight}.
Then the arrangement $\A$ is {\it weighted}:\  for $j\in J$, we assign weight $a_j$ to hyperplane $H_j$.
For an edge $X_\al$, define its weight
$a_\al = \sum_{j\in J_\al}a_j$.  We denote $a_J = {\sum}_{j\in J} a_j$ and 
  $\om^{(a)} = \sum_{j\in J} a_j\cdot (H_j)\ \in  \OS^1(\A)$.
 Multiplication by $\om^{(a)}$ defines the differential
 $ d^{(a)} :   \OS^p(\A) \to \OS^{p+1}(\A)$,  $x  \mapsto \om^{(a)}\cdot x$,
   on  $\OS(\A)$.

\subsection{Flag complex, see \cite{SV} }

 For an edge $X_\alpha$, \ $l_\alpha=p$, a {\it flag} starting at $X_\alpha$ is a sequence
$ X_{\alpha_0}\supset
X_{\alpha_1} \supset \dots \supset X_{\alpha_p} = X_\alpha $
of edges such that  $ l_{\alpha_j} = j$ for $j = 0, \dots , p$.
 For an edge $X_\alpha$,
 we define $(\overline{\FF}_{\alpha})_\Z$  as  the free $\Z$-module generated by the elements
$\overline{F}_{{\alpha_0},\dots,{\alpha_p}=\alpha}$
 la\-bel\-ed by the elements of
the set of all flags  starting at $X_\alpha$.
We define  $(\FF_{\alpha})_\Z$ as the quotient of
$(\overline{\FF}_{\alpha})_\Z$ by the submodule generated by all
the elements of the form
\bean
\label{rlF}
{\sum}_{X_\beta,\
X_{\alpha_{j-1}}\supset X_\beta\supset X_{\alpha_{j+1}}}\!\!\!
\overline {F}_{{\alpha_0},\dots,
{\alpha_{j-1}},{\beta},{\alpha_{j+1}},\dots,{\alpha_p}=\alpha}\, .
\eean
Such an element is determined  by  $j \in \{ 1, \dots , p-1\}$ and
an incomplete flag $X_{\alpha_0}\supset...\supset
X_{\alpha_{j-1}} \supset X_{\alpha_{j+1}}\supset...\supset
X_{\alpha_p} = X_\alpha$ with $l_{\alpha_i}$ $=$ $i$.

We denote by ${F}_{{\alpha_0},\dots,{\alpha_p}}$ the image in $(\FF_\alpha)_\Z$ of the element
$\overline{F}_{{\alpha_0},\dots,{\alpha_p}}$.  For $p=0,\dots,k$, we set
$({\FF}^p(\A))_\Z = \oplus_{X_\alpha,\, l_\alpha=p}\, ({\FF}_{\alpha})_\Z$,\
${\FF}^p(\A) = ({\FF}^p(\A))_\Z\otimes\C$,\
$\FF(\A) = \oplus_{p=1}^{N}\FF^p(\A)$.
We define the differential $d_\Z : (\FF^p(\A))_\Z \to (\FF^{p+1}(\A))_\Z$ by
\bean
\label{d in F}
d_\Z\  :\   {F}_{{\alpha_0},\dots,{\alpha_p}} \ \mapsto \
{\sum}_{X_{\beta},\
X_{\alpha_{p}}\supset X_{\beta}}
 {F}_{{\alpha_0},\dots,
{\alpha_{p}},{\beta}} ,
\eean
$d^2_\Z=0$. Tensoring $d_\Z$  with $\C$, we obtain the differential $d : \FF^p(\A) \to \FF^{p+1}(\A)$.
In particular, we have
$H^p(\FF(\A),d) = H^p((\FF(\A))_\Z,d_\Z)\otimes \C$.

We have  $H^p(\FF(\A),d)=0$ for $p\ne k$ and
$\dim H^k(\FF(\A),d)=|\chi(U(\A))|$, where $\chi(U(\A))$ is the Euler characteristic  of the complement $U(\A)$, see \cite[Corollary 2.8]{SV}

\subsection{Duality}
The vector spaces $\OS^p(\A)$ and $\FF^p(\A)$ are dual,  see \cite{SV}.
The pairing $ \OS^p(\A)\otimes\FF^p(\A) \to \C$ is defined as follows.
{}For $H_{j_1},...,H_{j_p}$ in general position, set
$F(H_{j_1},...,H_{j_p})=F_{{\alpha_0},\dots,{\alpha_p}}$, where
$X_{\alpha_0}=\C^k$, $X_{\alpha_1}=H_{j_1}$, \dots ,
$X_{\alpha_p}=H_{j_1} \cap \dots \cap H_{j_p}$.
Then we define $\langle (H_{j_1},...,H_{j_p}), F_{{\alpha_0},\dots,{\alpha_p}}
 \rangle = (-1)^{|\sigma|},$
if $F_{{\alpha_0},\dots,{\alpha_p}}
= F(H_{j_{\sigma(1)}},...,H_{j_{\sigma(p)}})$ for some $\sigma \in S_p$,
and $\langle (H_{j_1},...,H_{j_p}), F_{{\alpha_0},\dots,{\alpha_p}} \rangle = 0$ otherwise.

An element $F \in \FF^k(\A)$ is called  {\it singular}  if
$F$ annihilates the image of the map
$d^{(a)} :   \OS^{k-1}(\A) \to \OS^{k}(\A)$,  see \cite{V4}.
Denote by
$\on{Sing}_a\FF^k(\A) \subset \FF^k(\A)$  the subspace of all singular vectors.

\subsection{Contravariant map and form, see \cite{SV}}
 The weights  $a$ determines the {\it contravariant  map}
 \bean
 \label{S map}
  \mathcal S^{(a)} : \FF^p(\A) \to \OS^p(\A),
  \quad
  {F}_{{\alpha_0},\dots,{\alpha_p}} \mapsto
\sum  a_{j_1} \cdots a_{j_p} (H_{j_1}, \dots , H_{j_p})\,,
\eean
 where the sum is taken over all $p$-tuples $(H_{j_1},...,H_{j_p})$ such that
$H_{j_1} \supset X_{\al_1}$,\dots , $ H_{j_p}\supset X_{\alpha_p}$.
Identifying $\OS^p(\A)$ with $\FF^p(\A)^*$, we consider
the map  as a bilinear form,
$S^{(a)} : \FF^p(\A) \otimes \FF^p(\A) \to \C$.
The bilinear form is
called the {\it contravariant form}.
The contravariant form  is symmetric.
The contravariant map \Ref{S map} defines a homomorphism of complexes
$\mc S^{(a)} :(\FF(\A),d)\to (\OS(\A),d^{(a)})$, see {\cite[Lemma 3.2.5]{SV}}.

\subsection{Generic weights}
\label{secGnW}

\begin{thm}[{\cite[Theorem 3.7]{SV}}]
\label{thmShap}
If the weight $a$ is such that none of the dense edges has  weight zero, then
the contravariant form is nondegenerate. In particular, we have  an isomorphism
of complexes 
$\mc S :(\FF(\A),d)\to (\OS(\A),d^{(a)})$. 
\qed
\end{thm}

Notice that none of the dense edges has  weight zero if all weights are positive.

\smallskip
If the weight $a$ is such that none of the dense edges has  weight zero, then the isomorphism
of Theorem \ref{thmShap}
and the  graded algebra structure on $\OS(\A)$ induce a graded algebra structure on $\FF(\A)$.

\subsection{Differential forms}
\label{DifForms}

For  $j\in J$,  fix defining equations $f_j=0$ for the hyperplanes $H_j$,
where $f_j = b^1_jt_1+\dots+b^k_jt_k + z_j$
with $b^i_j, z_j\in \C$.
Consider the logarithmic differential 1-forms
$\omega_j = df_j/f_j$ on $\C^k$.
Let $\bar{\OS}(\A)$ be the exterior $\C$-algebra of differential forms
generated by 1 and $\omega_j$, $j\in J$.
The map ${\OS}(\A) \to \bar{\OS}(\A), \ (H_j) \mapsto \omega_j$,
is an isomorphism. We identify ${\OS}(\A)$ and $\bar{\OS}(\A)$.

For $I=\{i_1,\dots,i_k\}\subset J$, denote $d_I=d_{i_1,\dots,i_k}=\det_{i,l=1}^k(b^i_{i_l})$.
Then
$\om_{i_1}\wedge\dots\wedge \om_{i_k}=\frac{ d_{i_1,\dots,i_k}}{f_{i_1}\dots f_{i_k}}\,dt_{1}\wedge\dots\wedge dt_{k}$.

\subsection{Master function}
\label{function}

The {\it master function} of the weighted arrangement $(\A, a)$ is
\bean
\label{Mast k}
\Phi_{\A,a} = {\sum}_{j\in J}\,a_j \log f_j,
\eean
a multivalued function on $U(\A)$.
Let $C_{\A,a} =\{ u\in U(\A)\ |\ \frac{\der\Phi_{\A,a}}{\der t_i}(u)=0\ \on{for}\ i=1,\dots,k\}$ be the critical set of  $\Phi_{\A,a}$.

\subsection{Isolated critical points}
\label{IsCrPt}

For generic weight $a\in(\C^\times)^n$, all critical points of
$\Phi_{\A,a}$ are nondegenerate and the number of critical points equals $|\chi(U(\A))|$,
see \cite{V2, OT1, Si}.

  Consider the projective space $\Pee^k$ compactifying $\C^k$. Assign
the weight $a_\infty=-\sum_{j\in J} a_j$ to the hyperplane
$H_\infty=\Pee^k-\C^k$. Denote by  $\bar{\A}$
the arrangement $(H_j)_{j\in J\cup \infty}$ in $\Pee^k$.
The weighted arrangement $(\A, a)$ is called {\it unbalanced}
if the weight of any dense edge of $\bar{\A}$ is nonzero, see \cite{V5}.
For example, $(\A, a)$ is unbalanced if all weights $(a_j)_{j\in J}$ are positive.

 If $(\A, a)$ is unbalanced, then
all critical points of $\Phi_{\A,a}$  are isolated and the sum of their Milnor numbers equals
$|\chi(U(\A))|$, see {\cite[Section 4]{V5}}.

\subsection {Residue}

Let  $\mc O(U(\A))$ be the algebra of regular functions on $U(\A)$  and
\linebreak
$I_{\A,a} =\langle \frac{\partial \Phi_{\A,a}} {\partial t_i }
\ |\ i=1,\dots,k\ \rangle \subset \mc O(U(\A))$
the ideal generated by  first  derivatives of $\Phi_{\A,a}$.
Let $ \OC = \mc O(U(\A))/ I_{\A,a} $ be
the algebra of functions on the critical set  and
$[\,]: \mc O(U(\A)) \to  \OC$, $f\mapsto [f]$, the projection.
We assume that all critical points are isolated. In that case the  algebra
$ \OC$ is finite-dimensional and
 the elements  $[1/f_j]$, $j\in J$, generate $  \OC$ as an algebra, see {\cite[Lemma 2.5]{V5}}.

Let  $\mc R :  \OC \to \C$ be the Grothendieck residue,
\bea
[f] \ \mapsto \ \frac 1{(2\pi i)^k}\,\Res
\ \frac{ f}{\prod_{j=1}^k\, \frac{\der \Phi_{\A,a}}{\der t_j}}
=\frac{1}{(2\pi i)^k}\int_{\Gamma}
\frac{f\ dt_1\wedge\dots\wedge dt_k}{\prod_{j=1}^k \frac{\der \Phi_{\A,a}}{\der t_j}}\ .
\eea
Here  $\Gamma$ is the real $k$-cycle defined by the equations
$|\frac{\der \Phi_{\A,a}}{\der t_j}|=\epsilon_j,\ j=1,\dots,k$,
where  $\epsilon_j$ are  small positive numbers,
see \cite{GH}.
Define the   {\it residue bilinear form}  $(\,,\,)_{C_{\A,a}}$ on $\OC$ by
$([f],[g])_{C_{\A,a}} = \mc R([f][g])$.
This form is nondegenerate, see \cite{AGV}, and $([f][g],[h])_{C_{\A,a}}=([f],[g][h])_{C_{\A,a}}$ for all $[f],[g],[h]\in \OC$, thus  $(\OC, (\,,\,)_{C_{\A,a}})$ is a {\it Frobenius algebra}.

\subsection{Orthogonal projection}
\label{OProj}

Let
$\pi^\perp : \FF^k(\A)\to \on{Sing}_a\FF^k(\A)$
be the orthogonal projection with respect to $S^{(a)}$.

If the weight $a\in(\C^\times)^n$ is unbalanced, 
then $d \FF^{k-1}(\A)= \on{Sing}_a\FF^k(\A)^\perp$,
where $d \FF^{k-1}(\A)\subset \FF^k(\A)$ is the image of the differential defined by \Ref{d in F} and
$\on{Sing}_a\FF^k(\A)^\perp \subset \FF^k(\A)$ is the orthogonal complement to
 $\on{Sing}_a\FF^k(\A)$  with respect to $S^{(a)}$, see \cite[Lemma 2.14]{V8}.

Define the map
\bean
\label{psi map}
\nu_\A \  :\  \FF^k(\A) \to \OC,
\qquad
F \mapsto [f]\,,
\eean
where $f$ is defined by the formula
$\mc S^{(a)}(F) = f dt_1\wedge\dots\wedge dt_k$.
Clearly, $ \nu_\A(\on{Sing}_a\FF^k(\A)^\perp) = \nu_\A (d \FF^{k-1}(\A))=0$,
since $\om^{(a)}=0$ on $C_{\A,a}$.

\begin{thm} [\cite{V8}]
\label{1Mn thm}
If the weight $a\in(\C^\times)^n$ is unbalanced, then the  map
$\nu_\A\big|_{\on{Sing}_a\FF^k(\A)}  :  \on{Sing}_a\FF^k(\A) \to  \OC$
is an isomorphism of vector spaces.
The isomorphism $\nu_\A$ identifies the residue form on $\OC$ and
the contravariant form on $\sing \FF^k(\A)$ multiplied by $(-1)^k$,
$S^{(a)}(f,g) = (-1)^k
(\nu_\A(f),\nu_\A(g))_{C_{\A,a}}$
for  $f,g\in \on{Sing}_a\FF^k(\A).$
\qed
\end{thm}

\begin{rem}
If the weight $a\in(\C^\times)^n$ is unbalanced, then the isomorphism $\mu_\A$
 induces a commutative associative algebra structure
on $\sing_a \FF^k(\A)$. Together with the contravariant form  $S^{(a)}|_{\on{Sing}_a\FF^k}$  it is a Frobenius algebra.
The algebra of multiplication operators on $\sing_a \FF^k(\A)$ is 
called the {\it Bethe algebra}
 of the weighted arrangement $(\A,a)$.   This Bethe algebra is
an analog of the { Bethe algebra} in the theory of quantum integrable models,
see, for example, \cite{MTV1, MTV2, V4, V5}.

\end{rem}

\subsection{Integral structure on $\OC$ and $\Sing_a\FF^k(\A)$}
\label{IS}
If the weight $a$ is unbalanced, the formula $H^p(\FF(\A),d) = H^p((\FF(\A))_\Z,d_\Z)\otimes \C$
 and the isomorphism  $\nu_\A\big|_{\on{Sing}_a\FF^k(\A)} : H^k(\FF(\A),d)$
$ \cong \Sing_a\FF(\A) \to \OC$
define an integral structure on $\OC$.
More precisely, for   a $k$-flag of edges $X_{\al_0}\supset X_{\al_1}\supset\dots\supset X_{\al_k}$, let
$\mc S^{(a)}(F_{\al_0,\dots,\al_k}) = f_{\al_0,\dots,\al_k} dt_1\wedge\dots\wedge dt_k$.
Denote by $w_{\al_0,\dots,\al_k}$ the element $[f_{\al_0,\dots,\al_k}]\in \OC$.

\begin{cor} [\cite{V8}]
\label{m-span}
If the weight $a$ is unbalanced, then the
set of all elements $\{w_{\al_0,\dots,\al_k}\}$, labeled by all  $k$-flag of edges  of $\A$, spans
the vector space  $\OC$.
All linear relations between the elements of the set are corollaries of the relations
\bean
\label{rls w}
&&
{\sum}_{X_\beta,
X_{\alpha_{j-1}}\supset X_\beta\supset X_{\alpha_{j+1}}}\!\!\!
w_{{\alpha_0},\dots,
{\alpha_{j-1}},{\beta},{\alpha_{j+1}},\dots,{\alpha_p}=\alpha}=0\,,
\\
&&
{\sum}_{X_{\beta},
X_{\alpha_{p}}\supset X_{\beta}}
 w_{{\alpha_0},\dots,
{\alpha_{p}},{\beta}} = 0\,,
\notag
\eean
cf. formulas \Ref{rlF}, \Ref{d in F}.
\qed
\end{cor}

Similarly,  for   a $k$-flag of edges $X_{\al_0}\supset X_{\al_1}\supset\dots\supset X_{\al_k}$, let
$v_{\al_0,\dots,\al_k}$ be the orthogonal projection of $F_{\al_0,\dots,\al_k}$ to $\Sing_a\FF^k(\A)$.

\begin{cor} [\cite{V8}]
If the weight $a$ is unbalanced, then the
set of all elements $\{v_{\al_0,\dots,\al_k}\}$, labeled by all  $k$-flag of edges  of $\A$, spans
the vector space  $\Sing_a\FF^k(\A)$.
All linear relations between the elements of the set are corollaries of the relations
\bean
\label{relsV}
&&
{\sum}_{X_\beta,
X_{\alpha_{j-1}}\supset X_\beta\supset X_{\alpha_{j+1}}}\!\!\!
v_{{\alpha_0},\dots,
{\alpha_{j-1}},{\beta},{\alpha_{j+1}},\dots,{\alpha_p}=\alpha}=0\,,
\\
&&
{\sum}_{X_{\beta},
X_{\alpha_{p}}\supset X_{\beta}}
 v_{{\alpha_0},\dots,
{\alpha_{p}},{\beta}} = 0\,,
\notag
\eean
cf. formulas \Ref{rlF}, \Ref{d in F}.
\qed
\end{cor}

We have  $ \nu_\A : v_{\al_0,\dots,\al_k} \mapsto w_{\al_0,\dots,\al_k}$.
The elements $\{w_{\al_0,\dots,\al_k}\}\subset \OC$ and
$\{v_{\al_0,\dots,\al_k}\}\subset \Sing_a\FF^k(\A)$ are called the {\it
marked elements}. The relations \Ref{rls w}, \Ref{relsV}
are called the {\it marked relations}.

\subsection{Combinatorial connection, I}
\label{sec CCI}

Consider a deformation $\A(s)$ of the  arrangement $\A$, which preserves the combinatorics of $\A$.
Assume that  the edges of $\A(s)$ can be identified with the edges of $\A$ so that
the elements in formula  \Ref{rlF} and the differential in formula \Ref{d in F}
do not depend on $s$. 
Then for every $s$, the elements
$\{w_{\al_0,\dots,\al_k}(s)\}$ span $\mc O(C_{\A(s),a})$ as a vector space with linear relations
\Ref{rls w}  not depending on $s$. This  allows us to identify all the vector spaces $\mc O(C_{\A(s),a})$.
In particular, if an element $w(s)\in \mc O(C_{\A(s),a})$ is given, then the derivative $\frac{dw}{ds}$ is well-defined.
This construction is called the {\it combinatorial connection} on the family of algebras $\mc O(C_{\A(s),a})$, see \cite{V6}.
All the elements $\{w_{\al_0,\dots,\al_k}(s)\}$ are flat sections of the combinatorial connection.

Similarly we can define the combinatorial connection on the family of vector spaces
\linebreak
 $\Sing_{a}\FF^k(\A(s))$.

\subsection{Arrangement with normal crossings}

\label{NoCro}

 An essential arrangement $\A$ is {\it with normal crossings},
if exactly $k$ hyperplanes meet at every vertex of $\A$.
Assume that $\A$ is an essential arrangement with normal crossings.

A basis of $\OS^p(\A)$ is formed by
$(H_{j_1},\dots,H_{j_p})$, where
$\{{j_1} <\dots <{j_p}\}$  are independent ordered $p$-element subsets of
$J$. The dual basis of $\FF^p(\A)$ is formed by the corresponding vectors
$F(H_{j_1},\dots,H_{j_p})$.
These bases of $\OS^p(\A)$ and $\FF^p(\A)$  are called {\it standard}.
We  have
\bean
\label{skew}
F(H_{j_1},\dots,H_{j_p}) = (-1)^{|\sigma|}
F(H_{j_{\sigma(1)}},\dots,H_{j_{\sigma(p)}}), \qquad \on{for}\ \sigma \in \Si_p.
\eean
For an independent subset $\{j_1,\dots,j_p\}$, we have
$S^{(a)}(F(H_{j_1},\dots,H_{j_p}) , F(H_{j_1},\dots,H_{j_p})) = a_{j_1}\cdots a_{j_p}$
and
$S^{(a)}(F(H_{j_1},\dots,H_{j_p}) , F(H_{i_1},\dots,H_{i_k})) = 0$
for distinct elements of the standard basis.
If $a$ is unbalanced, then the marked elements in $\OC$ are
\bean
\label{markeD}
w_{i_1,\dots,i_k} = d_{i_1,\dots,i_k} \frac{a_{i_1}}{[f_{i_1}]}\dots \frac{a_{i_k}}{[f_{i_k}]}\,,
\eean
where  $\{i_1,\dots,i_k\}$ runs through the set of all independent $k$-element subsets of $J$. We have
$w_{i_{\sigma(1)},\dots,i_{\sigma(k)}} =
 (-1)^\sigma w_{i_1,\dots,i_k}$ for $\sigma \in \Si_k$.
We put $w_{i_1,\dots,i_k}=0$ if the set $\{i_1,\dots,i_k\}$ is dependent. The marked relations are labeled by independent subsets
$\{i_2,\dots,i_k\}$ and have the form
${\sum}_{j\in J} w_{j,i_2,\dots,i_k} = 0$.
The marked elements $v_{i_1,\dots,i_k}$ in $\Sing_a\FF^k(\A)$ are orthogonal projections to $\Sing_a\FF^k(\A)$ of
the elements $F(H_{i_1},\dots, H_{i_k})$ with the skew-symmetry property
$v_{i_{\sigma(1)},\dots,i_{\sigma(k)}} = (-1)^\sigma v_{i_1,\dots,i_k}$
for $\sigma \in \Si_k$
and the marked relations
${\sum}_{j\in J} v_{j,i_2,\dots,i_k} = 0$
labeled by independent subsets  $\{i_2,\dots,i_k\}$.

For any independent ordered subset $j_1,\dots,j_p \in J$ we denote
$F_{j_1,\dots,j_p}=F(H_{j_1},\dots,H_{j_p}) \in \FF^p(\A)$
and set $F_{j_1,\dots,j_p}=0$  if $j_1,\dots,j_p$ is a dependent subset. 

\begin{cor}
\label{corORT}
The orthogonal complement $\Sing_a\FF^k(\A)$ is generated by the elements
\\
$\sum_{j\in J}F_{j,i_1,\dots,i_{k-1}}$ labeled by independent subsets $\{i_1,\dots,i_{k-1}\}\in J$,
and an element of $ \FF^k(\A)$ lies in $\Sing_a\FF^k(\A)$ if and only if it is orthogonal to all 
the elements $\sum_{j\in J}F_{j,i_1,\dots,i_{k-1}}$.
\qed
\end{cor}

\section{Family of parallelly transported hyperplanes}
\label{secT Tr}

\subsection{Arrangement in  $\C^n\times\C^k$}

Consider $\C^k$ with coordinates $t_1,\dots,t_k$,\
$\C^n$ with coordinates $z_1,\dots,z_n$, the projection
$\tau :\C^n\times\C^k \to \C^n$.
Fix $n$ nonzero linear functions on $\C^k$,
$g_j=b_j^1t_1+\dots + b_j^kt_k,$\ $ j\in J,$
where $b_j^i\in \C$. Assume that the functions $\{g_j\}_{ j\in J}$,
span the dual space $(\C^k)^*$.

Define $n$ linear functions on $\C^n\times\C^k$,
$f_j = g_j + z_j,$ $ j\in J.$
Consider
 the arrangement of hyperplanes $\A  = \{ H_j\}_{j\in J}$  in $\C^n\times \C^k$, where $ H_j$ is the zero set of $f_j$,
and denote by $ U(\A ) = \C^n\times \C^k - \cup_{j\in J} H_j$ the complement.
For every $x=\in \C^n$, the arrangement $\A$
induces an arrangement $\A(x)$ in the fiber $\tau^{-1}(x)\cong \C^k$. 
 Then $\A(x)$ consists of
hyperplanes  $\{H_j(x)\}_{j\in J}$, defined in $\C^k$ by the equations
$g_j+x_j=0$. Thus $\{\A(x)\}_{x\in \C^n}$ is a family of arrangements in $\C^k$,
whose hyperplanes are transported parallelly to themselves as $x$ changes.
 Denote by $ U(\A(x)) = \C^k - \cup_{j\in J} H_j(x)$  the complement.
For almost all $x\in\C^k$ the arrangement $\A(x)$ is  
with normal crossings. The subset $\Delta\subset\C^n$ where
this does not hold, is a hypersurface  called the discriminant. On the discriminant see,
for example, \cite{BB, V5}.

\subsection{Combinatorial connection, II}
\label{sec Good}

For any
$x^1, x^2\in \C^n-\Delta$, the spaces $\FF^p(\A(x^1))$, $\FF^p(\A(x^2))$
 are canonically identified if a vector $F(H_{j_1}(x^1),\dots,H_{j_p}(x^1))$ of the first space
is identified  with the vector $F(H_{j_1}(x^2),\dots,H_{j_p}(x^2))$ of the second, in other words,
if we identify the standard bases of
these spaces.

Assume that a weight $a\in(\C^\times)^n$ is given. Then each arrangement $\A(x)$  is weighted.
The identification of spaces $\FF^p(\A(x^1))$,
$\FF^p(\A(x^2))$ for $x^1,x^2\in\C^n-\Delta$ identifies the corresponding subspaces
$\on{Sing}_a\FF^k(\A(x^1))$, $\on{Sing}_a\FF^k(\A(x^2))$ and contravariant forms.

Assume that the weighted arrangement  $(\A(x),a)$ is unbalanced  for some $x\in\C^n-\Delta$,
then $(\A(x),a)$ is unbalanced for all $x\in\C^n-\Delta$.
The identification of $\on{Sing}_a\FF^k(\A(x^1))$ and $\on{Sing}_a\FF^k(\A(x^2))$
also identifies the marked elements $v_{j_1,\dots,j_k}(x^1)$ and $v_{j_1,\dots,j_k}(x^2)$,
see Section \ref{NoCro}.
For  $x\in\C^n-\Delta$, denote $V=\FF^k(\A(x))$, $\on{Sing}_a V=\on{Sing}_a\FF^k(\A(x))$,  $v_{j_1,\dots,j_k} = v_{j_1,\dots,j_k}(x)$.
The triple $(V, \on{Sing}_a V, S^{(a)})$, with marked elements $v_{j_1,\dots,j_k}$,
 does not depend on  $x$ under the identification.

As a result of this reasoning we obtain the canonically trivialized
vector bundle
\bean
\label{cmbbndl}
\sqcup_{x\in \C^n-\Delta}\,\FF^k(\A(x))\to \C^n-\Delta,
\eean
with the canonically trivialized subbundle $\sqcup_{x\in \C^n-\Delta}\, \on{Sing}_a\FF^k(\A(x))\to \C^n-\Delta$
and the constant contravariant form on the fibers.
This trivialization identifies the bundle in \Ref{cmbbndl} with the bundle
$(\cd)\times V\to \cd$
and identifies 
 the subbundle  $\sqcup_{x\in \C^n-\Delta}\,\on{Sing}_a\FF^k(\A(x))\to \C^n-\Delta$ with the subbbundle
\bean
\label{cmbBundl}
(\cd)\times(\sv)\to \cd.
\eean
The bundle in \Ref{cmbBundl} is
 called the {\it combinatorial bundle}, the flat connection on it is called {\it combinatorial}, see Section
\ref{sec CCI} and
 \cite{V5, V6}.

\subsection{Gauss-Manin connection on  $(\cd)\times (\sv)\to \cd$}
\label{Constrn}

The {\it master function}  is
$\Phi_{\A,a} = {\sum}_{j\in J}\,a_j \log f_j$,
a multivalued function on $U( {\A})$. Let $\kappa\in\C^\times$.
The function $e^{\Phi_{\A,a}/\kappa}$
defines a rank one local system $\mc L_\kappa$
on $U(\A)$ whose horizontal sections
over open subsets of $U(\A)$
 are univalued branches of $e^{\Phi_{\A,a}/\kappa}$ multiplied by complex numbers,
see, for example, \cite{SV, V2}.
The vector bundle
\bea
\sqcup_{x\in \C^n-\Delta}\,H_k(U(\A(x)), \mc L_\kappa\vert_{U(\A(x))})
 \to  \C^n-\Delta
\eea
is called the {\it homology bundle}. The homology bundle has a canonical  flat Gauss-Manin connection.

For a fixed $x\in\C^n-\Delta$, choose  $\gamma\in H_k(U(\A(x)), \mc L_\kappa\vert_{U(\A(x))})$.
The linear map $
\{\gamma\} : \OS^k(\A(x)) \to \C$,  $\om \mapsto \int_{\gamma} e^{\Phi_{\A,a}/\kappa} \om$,
is an element of $\Sing_a\FF^k(\A(x))$ by Stokes' theorem.
It is known that for generic  $\kappa$
any element of $\Sing_a\FF^k(\A(x))$ corresponds to a certain
$\gamma$ and  in that case this construction  gives the {\it integration isomorphism}
\bean
\label{iSO}
H_k(U(\A(x)), \mc L_\kappa\vert_{U(\A(x))}) \to \on{Sing}_a\FF^k(\A(x)),
\eean
 see \cite{SV}. The precise values of $\kappa$, such that \Ref{iSO} is an isomorphism, can be deduced
 from the determinant formula in \cite{V1}.

For generic $\kappa$ the fiber isomorphisms \Ref{iSO} define an isomorphism of the homology bundle and
the combinatorial bundle \Ref{cmbBundl}. The  Gauss-Manin connection induces a  connection on the combinatorial bundle.
That connection on the combinatorial bundle is  also called the {\it Gauss-Manin connection}.

Thus, there are two connections on the combinatorial bundle: the combinatorial connection and the Gauss-Manin
connection depending on  $\kappa$. In this situation we  consider the differential equations
for flat sections of the Gauss-Manin connection with respect to the combinatorially flat standard basis. Namely,
let $\gamma(x) \in H_k(U(\A(x)), \mc L_\kappa\vert_{U(\A(x))})$ be a flat section of the Gauss-Manin connection.
Let us write the corresponding section $I_\gamma(x)$ of the bundle $(\cd)\times \Sing_a V\to \cd$
in the combinatorially flat standard basis,
$I_\gamma(x) =\!\!\!
\sum_{{\rm independent } \atop \{j_1 < \dots < j_k\} \subset J }\!\!\!
I_\gamma^{j_1,\dots,j_k}(x)
 F(H_{j_1}, \dots , H_{j_k})$,
$
 I_\gamma^{j_1,\dots,j_k}(x)
=
\int_{\gamma(x)} e^{\Phi_{\A,a}/\kappa}
 \omega_{j_1} \wedge \dots \wedge \omega_{j_k}.
$
We may rewrite it as $ I_\gamma(x) =
{\sum}_{{\rm independent } \atop \{j_1 < \dots < j_k\} \subset J }
I_\gamma^{j_1,\dots,j_k}(x)
v_{j_1, \dots , j_k}$ since  $ I_\gamma(x) \in \Sing_aV$. 
For $I=\sum I^{j_1,\dots,j_k}v_{j_1, \dots , j_k}$ and $j\in J$, we denote
$\frac{\der I}{\der z_j} = \sum \frac{\der I^{j_1,\dots,j_k}}{\der z_j}v_{j_1, \dots , j_k}$.
This formula defines the combinatorial connection on the combinatorial bundle.

 The  section $I_\ga$ satisfies the Gauss-Manin differential equations
\bean
\label{dif eqn}
\kappa \frac{\der I}{\der z_j}(x) = K_j(x)I(x),
\qquad
j\in J,
\eean
where  $K_j(x) \in \End(\Sing_aV)$. See a description of the operators $K_j(x)$, for example, in
\cite{OT2,V2,V5}.

\subsection{Critical set}
\label{Cr set}

Denote by $C_{\A,a}$ the critical set of $\Phi_{\A,a}$ in the $\C^k$-direction,
\bean
\label{crit C n}
C_{\A,a} =\Big\{ (x,u)\in U(\A)\subset \C^n\times\C^k\ \big|\ \frac{\der\Phi_{\A,a}}{\der t_i}(x,u)=0\ \on{for}\ i=1,\dots,k\Big\}.
\eean

Let 
$\mc O(C_{\A(x),a})$ be the algebra of regular functions on $C_{\A(x),a} = C_{\A,a}\cap \tau^{-1}(x)$. Namely,
for $x\in\C^n$, let $I_{\A(x),a}$ be the ideal in $\mc O(U(\A(x)))$ generated by
$\frac{\der\Phi_{\A,a}}{\der t_i}, i=1,\dots,k$. We set
$\mc O(C_{\A(x),a}) = \mc O(U(\A(x)))/I_{\A(x),a}$.
Assume that the weight $a$ is such that the pair $(\A(x),a)$ is unbalanced for some $x\in\cd$. Then
 we obtain the  vector {\it bundle of algebras}
$\sqcup_{x\in\C^n-\Delta} \OCx \to \C^n-\Delta$.
For $x\in\cd$, recall  the isomorphism
\bean
\label{fiber iso}  \nu(x):=
\nu_{\A(x)}\Big|_{\Sing_a\FF^k(\A(x)}  : \Sing_a\FF^k(\A(x) \to \OCx 
\eean
of Theorem \ref{1Mn thm}.
This \lq{fiber}\rq{} isomorphism establishes an isomorphism  of the bundle
\\
$\sqcup_{x\in\C^n-\Delta} \OCx \to \C^n-\Delta$ and the bundle
$(\cd)\times(\sv)\to \cd$.
This isomorphism together with  the combinatorial and Gauss-Manin connections on the bundle
$(\cd)\times(\sv)\to \cd$
induces two connections on the bundle of algebras
$\sqcup_{x\in\C^n-\Delta} \OCx \to \C^n-\Delta$,
 which also are called  the {\it combinatorial and Gauss-Manin connections}, respectively.

\smallskip

\begin{thm}[\cite{V5}]
\label{K/f}
If  the pair $(\A(x),a)$ is unbalanced for $x\in\cd$, then for all $j\in J$, we have
\bean
\left( \Big[\frac{a_j}{f_j}\Big] *_x \right) \circ 
\nu(x) = \nu(x)
 \circ K_j(x),
\eean
where $\big[\frac{a_j}{f_j}\big] *_x$ is the operator of multiplication by
$\big[\frac{a_j}{f_j}\big]$ on $\OCx$ and $ K_j(x)\in \End(\sing_aV)$ is the operator defined
in \Ref{dif eqn}.
\qed
\end{thm}

\begin{rem}
Recall that ${a_j/f_j} = {\der\Phi_{\A,a}/\der z_j}$ and the elements $ [{a_j/f_j}]$, $j\in J$, generate the algebra $\OCx$.
Theorem \ref{K/f} says that under the isomorphism $\nu(x)$ the operators of multiplication $[{a_j/f_j}]*_x$
on $\OCx$ are identified with the operators
$K_j(x)$ in the Gauss-Manin differential equations \Ref{dif eqn}. 
The correspondence of Theorem \ref{K/f}
defines
 a commutative algebra structure on $\sing_aV$, the structure depending on $x$. The multiplication in this commutative algebra is generated by
 the operators $K_j(x), j\in J$.  The correspondence of Theorem \ref{K/f}
  also defines the Gauss-Manin  differential equations on the bundle of algebras
in terms of the multiplication in the fiber algebras, see these differential equations in
 \cite[Theorem 3.9]{V8}.

\end{rem}

\subsection{Formulas for multiplication}
\label{Form multi} Recall that for $j\in J$, we denote
$p_j = \Big[ \frac{\der \Phi}{\der z_j}\Big]=\Big[\frac{a_j}{f_j}\Big]\,\in\, \OCx$.
 Then $w_{i_1,\dots,i_k} = d_{i_1,\dots,i_k} p_{i_1}\dots p_{i_k}$,
and for any $i_1,\dots,i_{k-1}\in J$ we have
\bean
\label{pREL}
{\sum}_{j\in J} d_{j,i_1,\dots,i_{k-1}}p_j =0.
\eean
For a subset $I=\{i_1,\dots,i_{k+1}\}\subset J$ denote
\bean
\label{f_I}
f_I(z)={\sum}_{j=1}^{k+1}(-1)^{j+1}z_{i_j}d_{i_1,\dots,\widehat{i_j}, \dots,i_{k+1}}.
\eean

\begin{lem}
\label{lem rel}
We have
\bean
\label{rEl}
f_I(z) {\prod}_{j=1}^{k+1} \frac{a_{i_j}}{f_{i_j}} = {\sum}_{j=1}^{k+1}(-1)^{j+1}
a_{i_j}d_{i_1,\dots,\widehat{i_j}, \dots,i_{k+1}}{\prod}_{m=1,\ m\ne j}^{k+1} \frac{a_{i_m}}{f_{i_m}}.
\eean

\end{lem}

\begin{proof}  We have $\sum_{j=1}^{k+1}(-1)^{j+1}(f_{i_j}-z_{i_j})d_{i_1,\dots,\widehat{i_j}, \dots,i_{k+1}}=0$. That implies the lemma.
\end{proof}

\begin{cor}
\label{cor prod}  Assume that a subset $\{i_1,\dots,i_{k+1}\}\subset J$ consists of distinct elements and contains a $k$-element independent subset. Then we have an identity in $\OCx$:
\bean
\label{propF}
f_{i_1,\dots,i_{k+1}}(z)\, p_{i_1}\dots p_{i_{k+1}} =  {\sum}_{j=1}^{k+1}(-1)^{j+1}
a_{i_j} w_{i_1,\dots,\widehat{i_j}, \dots,i_{k+1}}.
\eean

\end{cor}

\begin{proof}
 The corollary follows from Lemma \ref{lem rel}.
\end{proof}

\section{Elementary arrangements}
\label{EA}

\subsection{Definition}
\label{sdef}
Consider a $k$-dimensional vector space $X$ with coordinates $t_{1},t_{2},\dots,t_{k}$. 
Let $\lambda=(\lambda_{1},\lambda_{2},\dots,\lambda_{m})$ be a collection of positive integers
such that $\sum_{i=1}^{m}\lambda_{i}=k$. Assume that we have a collection of non-intersecting
 index sets $J_{h}$, $h=1,\dots,m$, each with $|J_{h}|=\lambda_{h}+1$ elements.
Denote $J_{\lambda}:=\cup_{h=1}^m J_h$,  $\la^h=\la_1+\dots+\la_h$, $\la^0=0$.

Let $\A_{J_\la} = \{H_i\}_{i\in J_\la} $ be a weighted 
arrangement of affine hyperplanes in $X$
with normal crossings.
 Each hyperplane $H_{i}$ has weight $a_i$ and is defined by an equation 
$g_{i}(t_{1},\dots,t_{k})+z_{i}=0$,
where $g_{i}=\sum_{l=1}^k b^l_it_l$ is an element of the dual space $X^{*}$ and $z_{i}\in\C$.

Define a subspace $X_{h}^{*}(J_{\lambda})=\mbox{span}\left\{ g_{i}\right\} _{i\in\bigcup_{j=1}^{h}J_{j}}$ of $X^*$. We have a filtration
\begin{equation}
\label{filt}
X_{1}^{\ast}(J_{\lambda})\subset X_{2}^{\ast}(J_{\lambda})\subset\cdots\subset X_{m}^{*}(J_{\lambda})\subset X^*.
\end{equation}

The arrangement $\mathcal{C}_{J_{\lambda}}$ is called \textit{elementary
of type $\lambda$} if 
\begin{enumerate}
\item[(i)] For any $h\in\{1,\dots,m\}$ we have $\dim X_{h}^{*}(J_{\lambda})=\la^h$.
\item[(ii)] For any $h\in\{1,\dots,m\}$, any $j\in J_{h}$ and
 $A:=\{ \cup_{i=1}^{h}J_{i}\} -\{ j\} $, we have
\\
$\dim(\on{span}\{ g_{l}\} _{l\in A})=\la^h$.
\end{enumerate}

\begin{figure}
  \begin{tikzpicture}
 
\draw[thick] (1,2) -- (-1,0) node[anchor=north] {$H_{2}$};
\draw[thick] (-0.5,2) -- (2,0) node[anchor=north] {$H_{3}$};
\draw[thick] (-1.2,.75) -- (2.5,.75) node[anchor=west] {$H_{1}$};

\draw[thick] (8.5,2) -- (7,0) node[anchor=north] {$H_{3}$};
\draw[thick] (7.5,2) -- (9.5,0) node[anchor=north] {$H_{4}$};
\draw[thick] (6.5,0.5) -- (9.5,0.5) node[anchor=west] {$H_{2}$};
\draw[thick] (6.5,1) -- (9.5,1) node[anchor=west] {$H_{1}$};
\end{tikzpicture}
 \caption{ 
Elementary Arrangements in Dimension 2.}
 
\end{figure}
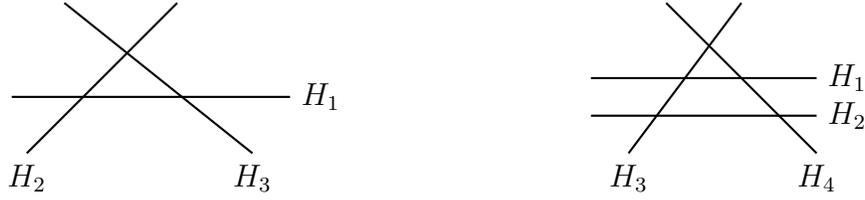

In Figure 2 the first elementary arrangement is of type $\la=(2)$ and the second is of type $\la=(1,1)$ with
$J_1=\{1,2\}, J_2=\{3,4\}$.
\smallskip
In this section we always assume that $\A_{J_\la}$ is an elementary arrangement.

\subsection{Distinguished Elements}
\label{DEl}

For $h=1,\dots,m$, let $K_{h}=\{j_1,\dots,j_{\la_h}\}$ be an ordered $\lambda_{h}$-element
subset of $J_{h}$. Recall the notation $F_{K_h} = F(H_{j_1},\dots,H_{j_{\la_h}})\in \F^{\la_h}(\A_{J_\la})$.
The elements of the flag space $\F^k(\A_{J_{\lambda}})$ of the form
\[
F_{K_{1},\dots,K_{m}}=F_{K_{1}}\wedge F_{K_{2}}\wedge\cdots\wedge F_{K_{m}}\in\mathcal{F}^{k}\left(\mathcal{C}_{J_{\lambda}}\right),
\]
are called the \textit{distinguished elements} of the elementary arrangement  $\mathcal{C}_{J_{\lambda}}$.

\begin{lem}
Counted up to permutation of indices potentially changing sign, there
are exactly $\prod_{h=1}^{m}(\lambda_{h}+1)$ distinguished elements
of $\mathcal{C}_{J_{\lambda}}$.
\qed
\end{lem}

For example, in Figure 2 the distinguished elements of the first arrangement are $\pm F_{1,2}, \pm F_{2,3},$ 
$\pm F_{1,3}$ 
and  distinguished elements of the second are $\pm F_{1,3}, \pm F_{1,4}, \pm F_{2,3}, \pm F_{2,4}$.

Let $J_{h}=\{j_{1}^{h},j_{2}^{h},\dots,j_{\lambda_{h}+1}^{h}\}$,
and let $K_{\widehat{j_{i}^{h}}}=\{j_{1}^{h},\dots,\widehat{j_{i}^{h}},\dots, j_{\lambda_{h}+1}^{h}\}$
be an ordered $\lambda_{h}$-element subset of $J_{h}.$ 
The element 
\bean
\label{sing}
s(\A_{J_{\lambda}})
&=&
\left({\sum}_{i=1}^{\lambda_{1}+1}(-1)^{i+1}a_{j_{i}^{1}}
F_{K_{\widehat{j_{i}^{1}}}}\right)\wedge\left({\sum}_{i=1}^{\lambda_{2}+1}
(-1)^{i+1}a_{j_{i}^{2}}F_{K_{\widehat{j_{i}^{2}}}}\right)
\wedge\dots
\\
\notag
&&
\phantom{aa} \dots
\wedge
\left({\sum}_{i=1}^{\lambda_{m}+1}(-1)^{i+1}a_{j_{i}^{m}}F_{K_{\widehat{j_{i}^{m}}}}\right)
\in\mathcal{F}^{k}(\mathcal{C}_{J_{\lambda}})
\eean
is called the \textit{singular element} of $\mathcal{C}_{J_{\lambda}}$.

\begin{lem}
\label{ssing}
We have $s(\A_{J_{\lambda}})\in \Sing_a \FF^k(\A_{J_\la})$.

\end{lem}

\begin{proof}
By Corollary \ref{corORT}, the element $s(\A_{J_{\lambda}})$ lies in $ \Sing_A \FF^k(\A_{J_\la})$
if and only if it is orthogonal to each of the elements 
${\sum}_{j\in J}F_{j,i_1,\dots,i_{k-1}}$ labeled by independent subsets $I=\{i_1,\dots,i_{k-1}\}\subset J_\la$.
If  $\{i_1,\dots,i_{k-1}\}\subset J_\la$ is independent, then there is $l\in\{1,\dots,m\}$ such that
$|I\cap J_l|=\la_l-1$  and $|I\cap J_h|=\la_h$ for $h\ne l$. Then it is clear that
$S^{(a)}(s(\A_{J_\la}), \sum_{j\in J}F_{j,i_1,\dots,i_{k-1}})=0$ as the sum of two opposite terms
coming from the $l$-th factor in formula \Ref{sing}.
\end{proof}

The singular element $s(\A_{J_{\lambda}})$ has the following properties. 
It is defined uniquely up to multiplication by $\pm1$, and this sign
depends on the ordering put on each subset $J_{h}$.  Each destinguished element of $\A_{J_\la}$, considered up to  sign, enters the singular element exactly once.

\begin{example}
 In Figure 2 the singular element of the first elementary arrangement is $\pm (a_3 F_{1,2}+a_2 F_{3,1} + a_1 F_{2,3}) $ 
and the singular element of the second is $\pm (a_2F_{1}- a_1F_2)\wedge(a_3F_4-a_4F_3)$.

\end{example}

\subsection{Decomposing Determinants}
\label{dec}

Recall the elements $g_j=\sum_{i=1}^{k}b_j^{i}t_i\in X^{*}$, used in defining the hyperplanes
$H_j$, and the notation $d_{j_{1},\dots,j_{k}}=\det\left(b_{j_{\ell}}^i\right)_{i,\ell=1}^{k}$.

Let $ s_{1},\dots,s_{k}$ be a basis of $X^*$  adjacent
to the filtration $\{X_{h}^{\ast}(J_{\lambda})\}_{h=1}^{m}$ 
in \Ref{filt}, i.e., let $ s_{1},\dots,s_{k}$ be such that for any
$h=1,\dots,m$, the elements
$s_{1},\dots,s_{\lambda^{h}} $ form a basis of $X_{h}^{\ast}(J_{\lambda})$.
Additionally, we select $ s_{1},\dots,s_{k} $ such
that the change of basis matrix from $ t_{1},\dots,t_{k} $
to $s_{1},\dots,s_{k}$ has determinant one.

Let $g_j=\sum_{i=1}^{k}c_j^is_i,$ be the expansion of $g_j$ with respect to the new basis, then
for any $j_1,\dots,j_k\in J_\la$ we have $\det(c_{j_{\ell}}^i)_{i,\ell=1}^{k} = d_{j_1,\dots,j_k}$ and
 $c_{j}^i=0$  for all $j\in J_{h}$,   $i>\lambda^{h}$.

For $h=1,\dots,m$, let $K_{h}=\{ k_{1}^{h},\dots,k_{\lambda_{h}}^{h}\} $
be an ordered $\lambda_{h}$-element subset of $J_{h}$. Consider $K=\left(K_{1},K_{2},\dots,K_{m}\right)$ as
an ordered $k$-element subset of $J_\la$. 
The matrix 
$(c_j^i)^{i=1,\dots,k}_{ j\in K}$
is lower block-triangular. Define the diagonal $\la_h\times\la_h$-blocks as the matrices
$C_{h}=(c_{k_{j}^{h}}^{i})_{j=1,\dots,\la_h}^{ i=\lambda^{h-1}+1,\dots, \lambda^{h}}$.

\begin{lem}
\label{ldec}
Let $K$ be constructed as above. Then, 
$d_{K}=\prod_{h=1}^{m}\det C_{h}$.
\qed
\end{lem}

\subsection{Auxiliary Arrangements}
\label{sAA}

For $h=1,\dots,m$, let $Y_{h}$ be a vector space of dimension $\lambda_{h}$
with coordinates $s_{1,h},\dots,s_{\lambda_{h},h}$. 
For $j\in J_h$, let
 $g_{j,h}$
be elements of $Y_{h}^{\ast}$ given by the formula $g_{j,h}=\sum_{i=1}^{\lambda_{h}}c_{j,h}^{i}s_{i,h}$
where $c_{j,h}^i:=c_{j}^{\lambda^{h-1}+i}$ and the  $c_{j}^{\lambda^{h-1}+i}$ are
the coefficients introduced in Section \ref{dec}.

 Define $\mathcal{C}_{J_{\lambda},h}=\{H_{j,h}\}_{j\in J_{h}}$
to be the following weighted arrangement of affine hyperplanes in $Y_{h}$. Each
hyperplane $H_{j,h}$ has weight $a_j$ and is defined by the equation $g_{j,h}(s_{1,h},\dots,s_{\lambda_{h},h})+z_j=0$,
where $z_j$ and $a_j$ are the same as in Section  \ref{sdef}.
We call $\mathcal{C}_{J_{\lambda},h}$ the \textit{auxiliary
arrangement of type ${h}$} associated with the elementary arrangement $\A_{J_\la}$.

For an ordered $\la_h$-element subset $I=\{j_1,\dots,j_{\la_h}\}\subset J_h$ denote
$d_{I;h} =\det(c_{j_l,h}^i)_{i,l=1}^{\la_h}$. Let $J_h=\{j_1^h,\dots,j_{\la_h+1}^h\}$.  
By the construction, the linear combination 
\\
${\sum}_{i=1}^{\la_h+1}(-1)^{i+1} d_{j_1,\dots,\widehat{j_i},\dots,j_{\la_h+1};h} g_{j_i^h}$ lies
in $X^*_{h-1}$. 
Choose some numbers $e_j, j\in \cup_{i=1}^{h-1}J_i$,  such that the linear combination
\bean
\label{lce}
{\sum}_{i=1}^{\la_h+1}(-1)^{i+1} d_{j_1,\dots,\widehat{j_i},\dots,j_{\la_h+1};h} g_{j_i^h}
+{\sum}_{ j\in \cup_{i=1}^{h-1}J_i}e_j g_j =0
\eean
as an element of the space $X^*$. Such numbers exist since $e_j, j\in \cup_{i=1}^{h-1}J_i$, span
$X^*_{h-1}$. Denote
\bean
\label{fh}
f_{\A_{J_\la,h}} = {\sum}_{i=1}^{\la_h+1}(-1)^{i+1} d_{j_1,\dots,\widehat{j_i},\dots,j_{\la_h+1};h} z_{j_i^h}
+{\sum}_{ j\in \cup_{i=1}^{h-1}J_i}e_j z_j.
\eean

We call the function
\bean
\label{1aux}
P_{\A_{J_{\lambda},h}}
=
\frac{\prod_{j\in J_{h}}a_{j}}{(2\lambda_{h})!}
\frac{(f_{\A_{J_{\lambda},h}})^{2\lambda_{h}}}
{(\prod_{i=1}^{\lambda_{h}+1}d_{j_1^h,\dots,\widehat{j_i^h},\dots,j_{\la_h+1}^h;h})^{2}}
\eean
 a {\it prepotential  of first kind} of the auxiliary arrangement $\A_{J_h}$.
We call the function
$P_{\A_{J_\la}} = \prod_{h=1}^m P_{\A_{J_{\lambda},h}}$
 a {\it prepotential  of first kind} of the elementary arrangement $\A_{J_\la}$.
We call the function
\bean
\label{2ell}
Q_{\A_{J_\la}} = \ln (f_{\A_{J_{\lambda},1}}) {\prod}_{h=1}^m P_{\A_{J_{\lambda},h}}
\eean
a {\it prepotential  of second kind} of the elementary arrangement $\A_{J_\la}$.
The prepotentials are not unique due to the choice of the numbers $e_j$ above.

\subsection{Elementary subarrangements} 
Let us return to the situation of Section \ref{secT Tr}. For $x\in \C^n-\Delta$ consider the weighted arrangement
$\A(x)$ with normal crossings.

Let $\A_{J_\la}(x)=\{H_i(x)\}_{i\in J_\la}$ be an elementary 
subarrangement of the arrangement $\A(x)$ of type $\la=(\la_1,\dots,\la_m)$. 
Recall that $J_\la=\cup_{h=1}^mJ_h \subset J$ with subsets $J_h$ satisfying properties described in Section 
\ref{sdef}. According to those properties if a subarrangement $\A_{J_\la}(x)=\{H_i(x)\}_{i\in J_\la}$ is an elementary 
subarrangement of $\A(x)$ for some
$x\in \C^n-\Delta$, then the subarrangement $\A_{J_\la}(x')=\{H_i(x')\}_{i\in J_\la}$, associated with the same $J_\la$,
 is an elementary 
subarrangement of $\A(x')$ for every $x'\in \C^n-\Delta$.

\begin{example}
If $\A(x)$ is a generic arrangement, then all elementary subarrangements are of type $\la=(k)$, they are given by $k+1$-element subsets of $J$.

\end{example}

\smallskip

Since $\A(x)$ is with normal crossings we have a natural embeddings of graded exterior algebras
$\FF(\A_{J_\la}(x))\subset \FF(\A(x))$ and an embedding of spaces
 $\Sing_A \FF^k(\A_{J_\la}(x)) \subset \Sing_A \FF^k(\A(x))$. In particular,
the singular element $s(\A_{J_\la}(x))$ of $\A_{J_\la}(x)$ can be considered as an element of $\Sing_A \FF^k(\A(x))$.

 Recall that for $h=1,\dots,m$, there are auxiliary
arrangements $\mathcal{C}_{J_{\lambda},h}(x)$ associated with $\mc C_{J_{\lambda}}(x)$.
For $h=1,\dots, m-1$, we define the {\it  weight of the auxiliary arrangement} $\mathcal{C}_{J_{\lambda},h}(x)$ with respect to
$\A(x)$ as the sum
$a(J_{\lambda}, J ,h)=\sum_{i\in J\, \on{such\,that}\,g_{i}\not\in X_{h}^{\ast}(J_{\lambda})}a_{i}$
and  the {\it weight of the elementary subarrangement} $\mc C_{J_{\lambda}}(x)$ with respect to
$\A(x)$ as  the product
$a(J_{\lambda},J)=a_J \cdot\prod_{h=1}^{m-1}a\left(J_{\lambda},J,h\right)$.
We define the {\it potentials of first and second kind} of the family of arrangements $\A(x)$, $x\in \C^n-\Delta$, to be respectively the following functions on $\C^n-\Delta$:
\bean
\label{P1}
P(x_1,\dots,x_n) &=&\sum
 \frac {1}{a(J_\la,J)}P_{\A_{J_\la}}(x_1,\dots,x_n),
\\
\label{P2}
Q(x_1,\dots,x_n) &=&\sum
\frac {a_J}{a(J_\la,J)} Q_{\A_{J_\la}}(x_1,\dots,x_n),
\eean
where the sums are over all elementary subarrangements $\A_{J_\la}(x)$ of $\mathcal{C}(x)$ and $P_{J_\la}(x), Q_{J_\la}(x)$ 
are the prepotentials of first and second kind, respectively,  of the elementary subarrangements $\A_{J_\la}(x)$ of the arrangement $\A(x)$.  The potentials are not uniques, since the prepotentials are not unique, see
Section \ref{sAA}.

\begin{example} The second arrangement in Figure 2  has three elementary subarrangements with $J_\la$ being
$\{1,3,4\}$ or $ \{2,3,4\}$ or $\{1,2\}\cup\{3,4\}$ of types $\la = (2), (2), (1,1)$, respectively.
The potential of second kind for that arrangement is
\bea
&&
Q(z_1,z_2,z_3, z_4) = a_1a_3a_4 \ln(f_{\{1,3,4\},1}(z_1,z_2,z_3))\frac{(f_{\{1,3,4\},1}(z_1,z_3,z_4))^4}{4!\,(d_{1,3}d_{3,4}d_{4,1})^2}
\\
&&
\phantom{a}
+ a_2a_3a_4 \ln(f_{\{2,3,4\},1}(z_2,z_3,z_4))\frac{(f_{\{2,3,4\},1}(z_2,z_3,z_4))^4}{4!\,(d_{2,3}d_{3,4}d_{4,2})^2}
\\
\notag
&&
\phantom{aa}
+ \frac{a_1a_2a_3a_4}{a_3+a_4} \ln(f_{\{1,2\}\cup\{3,4\},1}(z_1,z_2)) \frac {(f_{\{1,2\}\cup\{3,4\},1}(z_1,z_2))^2}{2!(d_{1;1}d_{2;1})^2}
 \frac {(f_{\{1,2\}\cup\{3,4\},2}(z_1,z_2,z_3,z_4))^2}{2!(d_{3;2}d_{4;2})^2} ,
\eea
c.f. this formula with the formula in the example of Section \ref{Sr}.

\end{example}

The potentials $P, Q$ for families of generic arrangements were constructed in \cite{V6}, c.f. \cite{V9}.

\begin{lem}
\label{lemINV}
For any independent subset $\{s_1,\dots,s_{k-1}\}\subset J$, we have
\bean
\label{INV}
{\sum}_{j\in J} d_{j,s_1,\dots,s_{k-1}}\frac{\der P}{\der z_j} = {\sum}_{j\in J} d_{j,s_1,\dots,s_{k-1}}\frac{\der Q}{\der z_j} =0.
\eean
\end{lem}

\begin{proof} It is enough to prove that $\sum_{j\in J} d_{j,s_1,\dots,s_{k-1}}\frac{\der P_{\A_{J_\la}(z)}}{\der z_j}
=\sum_{j\in J} d_{j,s_1,\dots,s_{k-1}}\frac{\der Q_{\A_{J_\la}(z)}}{\der z_j}=0$ for every elementary 
subarrangement $\A_{J_\la}(z)$ of $\A(x)$. To prove that, it is enough to prove that 
$\sum_{j\in J} d_{j,s_1,\dots,s_{k-1}}\frac{\der 
f_{\A_{J_\la,h}}}{\der z_j}=0$
 for any $h$, see formula  \Ref{fh}, but that is clear.
\end{proof}

\section{Orthogonal projection}
\label{OPR}

\subsection{Formula for orthogonal projection} 
\label{FOP}
Recall the objects  of Section \ref{secT Tr}.
For  $x\in\C^n-\Delta$, we denoted $V=\FF^k(\A(x))$, $\on{Sing}_a V=\on{Sing}_a\FF^k(\A(x))$, $F_{j_1,\dots,j_k} = F_{j_1,\dots,j_k}(x)$.
Let $\pi : V\to \Sing_a V$ be the orthogonal projection with respect to $S^{(a)}$.

For an ordered independent subset $I=\{i_1,\dots,i_k\}\subset J$, let $E_I$ be the set of all elementary subarrangements $\A_{J_\la}(x)$ of $\A(x)$ which have $F_I$ as a distinguished element. Let $\A_{J_\la}(x) \in E_I$ be such a subarrangement. Let $s(\A_{J_\la}(x))$ be the singular element of $\A_{J_\la}(x)$ considered as
an element of $\Sing_aV$. The singular element 
is defined up to multiplication by $\pm1$. We fix the sign so that the distinguished element $F_I$ enters
$s(\A_{J_\la}(x))$ with coefficient 1. 

\begin{thm}
\label{thmORT}
For an independent ordered subset $I=\{i_1,\dots,i_k\} \subset J$ we have
\bean
\label{FORT}
\pi(F_I) = {\sum}_{C_{J_{\lambda}}(x) \in E_I}\frac{1}{a(J_{\lambda},J)} s(\A_{J_{\lambda}}(x))\ \in \Sing_aV.
\eean
\end{thm}

\begin{cor}
\label{cGen} The space $\Sing _aV$ is generated by singular elements of elementary subarrangements.
\qed
\end{cor}

Notice that the singular element of an elementary subarrangement in $\C^k$ is a linear combination of at most $2^k$ basis vectors
$F_{l_1,\dots,l_k}\in V$, while the dimension of $V$ could be arbitrarily big  and grow with $n$.

\begin{example}  For the second arrangement in Figure 2,  we have
\bea
&&
\pi(F_{3,4})= \frac 1{a_1+a_2+a_3+a_4}\Big( (a_1F_{3,4}+ a_3F_{4,1}+a_4F_{1,3})
+ (a_2F_{3,4}+ a_3F_{4,2}+a_4F_{2,3})\Big),
\\
&&
\pi(F_{2,3})= \frac 1{a_1+a_2+a_3+a_4}\Big(
(a_4F_{2,3}+ a_2F_{3,4}+ a_3F_{4,2})
\\
&&
\phantom{aaaaaaaaaaaaaaaaaaaaaaaa}
+\frac 1{a_3+a_4} (a_1F_2-a_2F_1)\wedge (a_4F_3-a_3F_4)\Big),
\eea
where $(a_1F_2-a_2F_1)\wedge (a_4F_3-a_3F_4) = a_1a_4F_{2,3} - a_1a_3F_{2,4} -a_2a_4F_{1,3}+ a_1a_3F_{1,4}$.

\end{example} 

\subsection{Proof of Theorem \ref{thmORT}}
Recall that every element of the form $\sum_{j\in J} F_{j,l_1,\dots,l_{k-1}}$ is orthogonal to $\Sing_aV$.
 In order to construct $\pi(F_I)$ from $F_I$ we add to $F_I$ a linear combination of  elements of the form $\sum_{j\in J} F_{j,l_1,\dots,l_{k-1}}$
so that the result is a linear combination of the singular elements of elementary subarrangements  $\A_{J_\la}(x)$.
That means that the result lies in $\Sing_aV$ by Lemma \ref{ssing}.

The transition from  $F_I$ to $\pi(F_I)$ is done in $k$ steps and this  reasoning is
by induction on the number $m$ appearing in the presentation $\la=(\la_1,\dots,\la_m)$. As the first step
 we add to $F_I$ a linear combination of elements of the form $\sum_{j\in J} F_{j,l_1,\dots,l_{k-1}}$ and transform $F_I$ to the sum
\\
$ \sum_{C_{J_{\lambda}}(x) \in E_I \,\on{with}\, m=1}\frac{1}{a(J_{\lambda},J)} s(\A_{J_{\lambda}}(x)) + R_1$, where
$R_1$ is a remainder. Then we add a new linear combination of elements of the form $\sum_{j\in J} F_{j,l_1,\dots,l_{k-1}}$ and transform the result to the sum $ \sum_{C_{J_{\lambda}}(x) \in E_I \,\on{with}\, m\leq 2}\frac{1}{a(J_{\lambda},J)} s(\A_{J_{\lambda}}(x)) + R_2$, and
so on.  After $m$ steps the result will be the right-hand side in \Ref{FORT} and there will be no remainder.

We illustrate that reasoning by considering the case $k=3$.  We construct the orthogonal projection of the element $F_{1,2,3}$,
 which could be an arbitrary basis vector of $V$ after reordering  hyperplanes.  Formula \Ref{FORT} says  
\bean
\label{3d orth}
\pi(F_{1,2,3}) &=& \Si_{1,2,3}\ +\ \Si_{1,2;3} + \Si_{1,3;2} + \Si_{2,3;1}\  +\ \Si_{1;2,3}+ \Si_{2;1,3}+\Si_{3;1,2}
+
\\
&+&
 \Si_{1;2;3} + \Si_{1;3;2} + \Si_{2;1;3}+ \Si_{2;3;1} + \Si_{3;1;2} + \Si_{3;2;1},
\notag
\eean
where
\bea
\Si_{1,2,3} := \frac 1{a_J}\sum{}^{1,2,3} (a_jF_{1,2,3}-a_1F_{j,2,3} + a_2F_{1,j,3}-a_3F_{1,2,j}),
\eea
\bea
\Si_{1,2;3} := \frac {1}{a_Ja(1,2)}\sum{}^{1,2}(a_jF_{1,2}-a_1F_{j,2} + a_2F_{j,1})\wedge \sum{}^{1,2;3}(a_hF_3 - a_3F_h),
\eea
\bea
\Si_{1,3;2} := \frac {-1}{a_Ja(1,3)}\sum{}^{1,3}(a_jF_{1,3}-a_1F_{j,3} + a_2F_{j,1})\wedge \sum{}^{1,3;2}(a_hF_2 - a_2F_h),
\eea
\bea
\Si_{2,3;1} := \frac {1}{a_Ja(2,3)}\sum{}^{2,3}(a_jF_{2,3}-a_2F_{j,3} + a_3F_{j,2})\wedge \sum{}^{2,3;1}(a_hF_1 - a_1F_h),
\eea
\bea
\Si_{1;2,3} := \frac {1}{a_Ja(1)}\sum{}^{1}
(a_jF_{1}-a_1F_{j})\wedge \sum{}^{1;2,3} (a_hF_{2,3}  - a_2F_{h,3} + a_3 F_{h,2}),
\eea
\bea
\Si_{2;1,3} := \frac {-1}{a_Ja(2)}
\sum{}^{2}
(a_jF_{2}-a_2F_{j})\wedge \sum{}^{2;1,3} (a_hF_{1,3}  - a_1F_{h,3} + a_3 F_{h,1}),
\eea
\bea
\Si_{3;1,2} :=  \frac {1}{a_Ja(3)}
\sum{}^{3}
(a_jF_{3}-a_3F_{j})\wedge \sum{}^{3;1,2} (a_hF_{1,2}  - a_1F_{h,2} + a_2 F_{h,1}),
\eea
\bea
\Si_{1;2;3}:
=
 \frac {1}{a_Ja(1)a(1,2)}
\sum{}^{1}
(a_jF_{1}-a_1F_{j})\wedge \sum{}^{1;2} (a_hF_{2}  - a_2F_{h})\wedge \sum{}^{1,2;3}(a_iF_3-a_3F_i),
\eea
\bea
\Si_{1;3;2}:
=
 \frac {-1}{a_Ja(1)a(1,3)}
\sum{}^{1}
(a_jF_{1}-a_1F_{j})\wedge \sum{}^{1;3} (a_hF_3  - a_3F_h)\wedge \sum{}^{1,3;2}(a_iF_2-a_2F_i),
\eea
\bea
\Si_{2;3;1}: =
 \frac {1}{a_Ja(2)a(2,3)}
\sum{}^{2}
(a_jF_{2}-a_2F_{j})\wedge \sum{}^{2;3} (a_hF_3  - a_3F_h)\wedge \sum{}^{2,3;1}(a_iF_1-a_1F_i),
\eea
\bea
\Si_{2;1;3}:
=
 \frac {-1}{a_Ja(2)a(2,1)}
\sum{}^{2}
(a_jF_{2}-a_2F_{j})\wedge \sum{}^{2;1} (a_hF_1  - a_1F_h)\wedge \sum{}^{2,1;3}(a_iF_3-a_3F_i),
\eea
\bea
\Si_{3;2;1}:
=
 \frac {-1}{a_Ja(3)a(2,3)}
\sum{}^{3}
(a_jF_{3}-a_3F_{j})\wedge \sum{}^{3;2} (a_hF_2  - a_2F_h)\wedge \sum{}^{3,2;1}(a_iF_1-a_1F_i),
\eea
\bea
\Si_{3;1;2}:
=
 \frac {1}{a_Ja(3)a(1,3)}
\sum{}^{3}
(a_jF_{3}-a_3F_{j})\wedge \sum{}^{3;1} (a_hF_1  - a_1F_h)\wedge \sum{}^{3,1;2}(a_iF_2-a_2F_i).
\eea
In these  formulas we use the following notations.

We denote
$a(h,l) = \sum a_j$, where the sum is over all $j\in J$ such that $g_j\notin \on{span}(g_h,g_l)$.  We
denote
$a(h) = \sum a_j$, where the sum is over all $j\in J$ such that $g_j\notin \on{span}(g_h)$.

 The sum $\sum{}^{1,2,3}$ is over all $j\in J$ such that the subset $\{j,1,2,3\}$ forms a circuit in $J$.
The sum $\sum^{h,l}$ is over all $j\in J$ such that the subset $\{j,h,l\}$ forms a circuit in $J$.
The  sum $\sum^h$  is over all $j\in J$ such that the subset $\{j,h\}$ forms a circuit in $J$.

The sum $\sum^{j,l;s}$ is over all $h\in J$ such that $g_h \notin\on{span}(g_j,g_l)$.
The sum $\sum^{j;l,s}$ is over all $h\in J$ such that $\on{span}(g_j,g_l, g_h)
=\on{span}(g_j,g_s, g_h) =(\C^3)^*$. 
The sum $\sum^{j;l}$ is over all  $h\in J$ such that $ \on{span}(g_j,g_l)=\on{span}(g_j,g_h)$.

\smallskip

The first transformation is 
\bean
\label{st1}
F_{1,2,3} \mapsto  F_{1,2,3} - \frac{a_1}{a_J}{\sum}_{j\in I} F_{j,2,3}  - \frac{a_2}{a_J}
{\sum}_{j\in I} F_{1,j,3}
- \frac{a_3}{a_J}{\sum}_{j\in I} F_{1,2,j},
\eean
the added terms are a linear combination of elements of the form $\sum_{j\in J}F_{j,l_1,l_2}$. We rearrange
 the right-hand side of \Ref{st1} as follows:
\bean
\label{st12}
&&
=
a_J^{-1}\sum{}^{1,2,3}(a_jF_{1,2,3}-a_1F_{j,2,3} + a_2F_{1,j,3}-a_3F_{1,2,j})
\\
\label{st13}
&&
+a_J^{-1}\sum{}^{1,2}(a_jF_{1,2}-a_1F_{j,2} + a_2F_{j,1})\wedge F_3 
\\
\label{st14}
&&
-a_J^{-1}\sum{}^{1,3}(a_jF_{1,3}-a_1F_{j,3} + a_3F_{j,1})\wedge F_2
\\
\label{st15}
&&
+a_J^{-1}\sum{}^{2,3}(a_jF_{2,3}-a_2F_{j,3} + a_3F_{j,2})\wedge F_1 
\\
\label{st15}
&&
+a_J^{-1}\sum{}^1(a_jF_{1}-a_1F_{j})\wedge F_{2,3}
-a_J^{-1}\sum{}^2(a_jF_{2}-a_2F_{j})\wedge F_{1,3}
\\
\label{st16}
&&
+a_J^{-1}\sum{}^3(a_jF_{3}-a_3F_{j})\wedge F_{1,2}.
\eean
The sum in \Ref{st12} is exactly the sum $ \sum_{C_{J_{\lambda}}(x) \in E_{\{1,2,3\}} \,\on{with}\, m=1}\frac{1}{a(J_{\lambda},J)} s(\A_{J_{\lambda}}(x))$ and
 the sums in \Ref{st13}-\Ref{st16} form the first remainder $R_1$. 

Now we add to each of the sums in  \Ref{st13}-\Ref{st16} a linear combination of elements of the form $\sum_{j\in J}F_{j,l_1,l_2}$ as follows.
Let 
\bea
\notag
&&
a_J^{-1}\sum{}^{1,2}(a_jF_{1,2}-a_1F_{j,2} + a_2F_{j,1})\wedge F_3 \to
\\
&&
a_J^{-1}\sum{}^{1,2}(a_jF_{1,2}-a_1F_{j,2} + a_2F_{j,1})\wedge \left( F_3 - \frac {a_3}{a(1,2)}{\sum}_{i\in J} F_i\right)
\\
&&
=\frac 1{a_J a(1,2)}\sum{}^{1,2}(a_jF_{1,2}-a_1F_{j,2} + a_2F_{j,1})\wedge {\sum}^{1,2;3}(a_i F_3 - a_3F_i).
\eea
 Similarly,  
\bea
\notag
&&
-a_J^{-1}\sum{}^{1,3}(a_jF_{1,3}-a_1F_{j,3} + a_3F_{j,1})\wedge F_2 \to
\\
&&
-\frac 1{a_J a(1,3)}\sum{}^{1,3}(a_jF_{1,3}-a_1F_{j,3} + a_2F_{j,1})\wedge {\sum}^{1,3;2}(a_i F_2 - a_2F_i),
\eea
\bea
\notag
&&
a_J^{-1}\sum{}^{2,3}(a_jF_{2,3}-a_2F_{j,3} + a_3F_{j,2})\wedge F_1 \to
\\
&&
\frac1{a_J a(2,3)}\sum{}^{2,3}(a_jF_{2,3}-a_2F_{j,3} + a_3F_{j,2})\wedge {\sum}^{2,3;1}(a_iF_1-a_1F_i).
\eea
  Similarly
\bea
&&
a_J^{-1}\sum{}^1(a_jF_{1}-a_1F_{j})\wedge F_{2,3} \to
\\
&&
a_J^{-1}\sum{}^1(a_jF_{1}-a_1F_{j})\wedge \Big(
 F_{2,3} - \frac{a_2}{a(1)} {\sum}_{i\in J} F_{i,3}
-\frac{a_3}{a(1)} {\sum}_{i\in J} F_{2,i} \Big)
\\
&&
=\frac{1}{a_J a(1)} \sum{}^1(a_jF_{1}-a_1F_{j})\wedge \sum{}^{1;2,3}(a_iF_{2,3}  - a_2F_{i,3} + a_3 F_{i,2})
\eea
\bea
&&
+ \frac{1}{a_J a(1)} \sum{}^1(a_jF_{1}-a_1F_{j})\wedge  \sum{}^{1;2}(a_iF_2-a_2F_i)\wedge F_3 
\\
&&
-\frac{1}{a_J a(1)} \sum{}^1(a_jF_{1}-a_1F_{j})\wedge  \sum{}^{1;3}(a_iF_3-a_3F_i)\wedge F_2. 
\eea
We transform similarly the remaining two sums in \Ref{st15}-\Ref{st16}.  This finishes step two of the procedure. After the two steps the result 
is
\bean
\label{st31}
&&
 \Si_{1,2,3}\ +\ \Si_{1,2;3} + \Si_{1,3;2} + \Si_{2,3;1}\  +\ \Si_{1;2,3}+ \Si_{2;1,3}+\Si_{3;1,2}
\\
\label{st32}
&&
+ \frac{1}{a_J a(1)} \sum{}^1(a_jF_{1}-a_1F_{j})\wedge  \sum{}^{1;2}(a_iF_2-a_2F_i)\wedge F_3 
\\
&&
\label{st33}
-\frac{1}{a_J a(1)} \sum{}^1(a_jF_{1}-a_1F_{j})\wedge  \sum{}^{1;3}(a_iF_3-a_3F_i)\wedge F_2  
\eean
\bean
\label{st34}
&&
- \frac{1}{a_J a(2)} \sum{}^2(a_jF_{2}-a_2F_{j})\wedge  \sum{}^{2;1}(a_iF_1-a_1F_i)\wedge F_3 
\\
&&
\label{st35}
+\frac{1}{a_J a(2)} \sum{}^2(a_jF_{2}-a_2F_{j})\wedge  \sum{}^{2;3}(a_iF_3-a_3F_i)\wedge F_1  
\eean
\bean
\label{st36}
&&
- \frac{1}{a_J a(3)} \sum{}^3(a_jF_{3}-a_3F_{j})\wedge  \sum{}^{3;2}(a_iF_2-a_2F_i)\wedge F_1 
\\
&&
\label{st37}
+\frac{1}{a_J a(3)} \sum{}^3(a_jF_{3}-a_3F_{j})\wedge  \sum{}^{3;1}(a_iF_1-a_1F_i)\wedge F_2.  
\eean
The sum in \Ref{st31} is the sum $ \sum_{C_{J_{\lambda}}(x) \in E_{\{1,2,3\}} \,\on{with}\, m\leq 2}\frac{1}{a(J_{\lambda},J)} s(\A_{J_{\lambda}}(x))$ and
 the sums in \Ref{st32}-\Ref{st37} form the second remainder $R_2$. As the third  step we transform the expression in \Ref{st32} to 
\bea
&&
 \frac{1}{a_J a(1)} \sum{}^1(a_jF_{1}-a_1F_{j})\wedge  \sum{}^{1;2}(a_iF_2-a_2F_i)\wedge (F_3 - \frac{a_3}{a(1,2)}{\sum}_{l\in J}F_l)
\\
&&
= \frac{1}{a_J a(1)a(1,2)} \sum{}^1(a_jF_{1}-a_1F_{j})\wedge  \sum{}^{1;2}(a_iF_2-a_2F_i)\wedge \sum{}^{1,2;3}(a_lF_3 - a_3F_l),
\eea
and similarly we transform the expressions in \Ref{st33}-\Ref{st37}.  After these three steps we obtain formula \Ref{3d orth}.
The case of arbitrary $k$ is similar to this case of $k=3$. Theorem \ref{thmORT} is proved.
\qed

\section {Potential of first kind}
\label{PT1}

Recall the objects  of Section \ref{secT Tr}.
Recall that $v_{j_1,\dots,j_k}=\pi(F_{j_1,\dots,j_k})$, where
 $\pi : V\to \Sing_a V$ is the orthogonal projection with respect to $S^{(a)}$. Let $P$ be the potential of first kind of the family of arrangements
$\A(x),x\in \C^n-\Delta$.

\begin{thm}
\label{thmP1}
For any two ordered independent subsets $I=\{i_1,\dots,i_k\}, L=\{l_1,\dots,l_k\} \subset J$, we have
\bean
\label{P1F}
S^{(a)}(v_{i_1,\dots,i_k}, v_{l_1,\dots,l_k})
=
 d_{i_1,\dots,i_k} d_{l_1,\dots,l_k}
 \frac{\der^{2k}P}{\der z_{i_1}\dots\der z_{i_k}\der z_{l_1}\dots\der z_{l_k}}.
\eean
\end{thm}

For families of generic arrangements this theorem was proved in \cite{V6}, c.f. \cite{V9}.

\begin{proof}  Since $ v_{l_1,\dots,l_k} = \pi(F_{l_1,\dots,l_k})$ we have
\bea
\notag
S^{(a)}(v_{i_1,\dots,i_k}, v_{l_1,\dots,l_k})
&=&
S^{(a)}(v_{i_1,\dots,i_k}, F_{l_1,\dots,l_k})
\\
&=&
{\sum}_{C_{J_{\lambda}}(x) \in E_{i_1,\dots, i_k}}\frac{1}{a(J_{\lambda},J)} S^{(a)}( s(\A_{J_{\lambda}}(x)), F_{l_1,\dots,l_k}).
\eea
Formula \Ref{sing} for the singular element $s(\A_{J_{\lambda}}(x))$ shows that the number $S^{(a)}( s(\A_{J_{\lambda}}(x)), F_{l_1,\dots,l_k})$
is nonzero if and only if $F_{l_1,\dots,l_k}$ is a distinguished elements of the elementary arrangement $\A_{J_{\lambda}}(x)$. Let this condition be satisfied for an elementary arrangement $\A_{J_{\lambda}}(x) \in E_{i_1,\dots, i_k}$. Let $J_\la =\cup_{h=1}^mJ_h$,
$J_{h}=\{j_{1}^{h},j_{2}^{h},\dots,j_{\lambda_{h}+1}^{h}\}$,
and let $K_{\widehat{j_{i}^{h}}}=\{j_{1}^{h},\dots,\widehat{j_{i}^{h}},\dots, j_{\lambda_{h}+1}^{h}\}$
be an ordered $\lambda_{h}$-element subset of $J_{h}.$  Then
\bean
\label{sing1}  
s(\A_{J_{\lambda}}(x))
&=&
\left({\sum}_{i=1}^{\lambda_{1}+1}(-1)^{i+1}a_{j_{i}^{1}}
F_{K_{\widehat{j_{i}^{1}}}}\right)\wedge\left({\sum}_{i=1}^{\lambda_{2}+1}
(-1)^{i+1}a_{j_{i}^{2}}F_{K_{\widehat{j_{i}^{2}}}}\right)
\wedge\dots
\\
\notag
&&
\phantom{aa} \dots
\wedge
\left({\sum}_{i=1}^{\lambda_{m}+1}(-1)^{i+1}a_{j_{i}^{m}}F_{K_{\widehat{j_{i}^{m}}}}\right).
\eean
Due to the choice of sign of $s(\A_{J_{\lambda}}(x))$ in Section \ref{FOP} we may assume that
\bean
\label{1F}
F_{i_1,\dots,i_k} = F_{K_{\widehat{j_{1}^{1}}}}\wedge F_{K_{\widehat{j_{1}^{2}}}}\wedge\dots F_{K_{\widehat{j_{1}^{m}}}}.
\eean
We may also assume that 
\bean
\label{2F}
F_{l_1,\dots,l_k} = F_{K_{\widehat{j_{s_1}^{1}}}}\wedge F_{K_{\widehat{j_{s_2}^{2}}}}\wedge\dots F_{K_{\widehat{j_{s_m}^{m}}}}
\eean
for some $s_h\in \{1,\dots,\la_h+1\}, h=1,\dots,m$. ${}^{(1)}$

{\let\thefootnote\relax
\footnotetext{\vsk-.8>\noindent
$^{(1)}\<$ Notice that if the indices of $F_{l_1,\dots,l_k} $ are permuted then
 $F_{l_1,\dots,l_k}$ is multiplied by $\pm1$ and  $d_{l_1,\dots,l_k}$ in the right-hand side of \Ref{P1F} is multiplied by the same $\pm1$.
}}

Equations \Ref{sing1}-\Ref{2F} imply
\bean
\label{S(sF)}
S^{(a)}( s(\A_{J_{\lambda}}(x)), F_{l_1,\dots,l_k}) = {\prod}_{h=1}^m (-1)^{s_h+1}
{\prod}_{h=1}^m
{\prod}_{j\in J_h}a_j.
\eean

Recall  formula \Ref{P1} for the potential of first kind. It is a linear combination of prepotentials 
of first kind  $P_{\A_{J_\la}}(x)$ of  all elementary subarrangements $\A_{J_\la}(x)$ of $\mathcal{C}(x)$. To finish the proof of Theorem \ref{thmP1} we need to show that if 
$\A_{J_\la}(x)$  is as in formula \Ref{S(sF)}, then
\bean
\label{derg}
 d_{i_1,\dots,i_k} d_{l_1,\dots,l_k}
 \frac{\der^{2k}P_{\A_{J_\la}}}{\der z_{i_1}\dots\der z_{i_k}\der z_{l_1}\dots\der z_{i_k}}={\prod}_{h=1}^m (-1)^{s_h+1}{\prod}_{h=1}^m{\prod}_{i\in J_h}a_i
\eean
and
\bean
\label{derb}
 \frac{\der^{2k}P_{\A_{J_\la}}}{\der z_{i_1}\dots\der z_{i_k}\der z_{l_1}\dots\der z_{i_k}}=0
\eean
if $\A_{J_{\lambda}}(x) \not \in E_{i_1,\dots, i_k}$ or if $\A_{J_{\lambda}}(x) \in E_{i_1,\dots, i_k}$ but $F_{l_1,\dots,l_k}$ is not a distinguished element of $\A_{J_{\lambda}}(x)$.

Recall the formula
\bea
P_{\A_{J_\la}}
=
{\prod}_{h=1}^m
\frac{\prod_{j\in J_{h}}a_{j}}{\left(2\lambda_{h}\right)!}
\frac{(f_{\A_{J_\la,h}})^{2\la_h}}
{(\prod_{i=1}^{\lambda_{h}+1}d_{j_1^h,\dots,\widehat{j_i^h},\dots,j_{\la_h+1}^h;h})^{2}}.
\eea

\begin{lem}
\label{lemH}
The $2\la_h$-th derivative of 
$\frac{1}{(2\lambda_{h})!}
\frac{(f_{\A_{J_\la,h}})^{2\la_h}}
{(\prod_{i=1}^{\lambda_{h}+1}d_{j_1^h,\dots,\widehat{j_i^h},\dots,i_{\la_h+1}^h;h})^{2}}
$
with respect to the variables $z_{j^h_2},\dots, z_{j^h_{\la_h+1}}, z_{j^h_1},\dots, \widehat{z_{j^h_{s_h}}},\dots,
z_{j^h_{\la_h+1}}$ equals 
$
\frac{(-1)^{s_h+1}}{ d_{j_2^h,\dots,j_{\la_h+1}^h;h} d_{j_1^h,\dots,\widehat{j_{s_h}^h},\dots,j_{\la_h+1}^h;h}}.
$
\qed
\end{lem}

Now Lemmas \ref{lemH} and \ref{ldec} imply formula \Ref{derg}. 

\begin{lem}
\label{lem0}
Formula \Ref{derb}  holds if $\A_{J_{\lambda}}(x) \not \in E_{i_1,\dots, i_k}$ or if $\A_{J_{\lambda}}(x) \in E_{i_1,\dots, i_k}$ but $F_{l_1,\dots,l_k}$ is not a distinguished element of $\A_{J_{\lambda}}(x)$.  
\end{lem}

\begin{proof}  Let $\A_{J_\la}$ be any elementary subarrangement
and $F_{i_1,\dots,i_k}, F_{l_1,\dots,l_k}$ two nonzero elements. 
Clearly formula  \Ref{derb}   holds if $\{i_1,\dots,i_k,$ $l_1,\dots,l_k\}\not\subset
\cup_{h=1}^mJ_h$.  Assume that $\{i_1,\dots,i_k,$ $l_1,\dots,l_k\}\subset
\cup_{h=1}^mJ_h$.  For $h=1,\dots,m$ denote $i^h=|\{s \ |\ i_s \in  J_h\}|$ and
$l^h=|\{s \ |\ l_s \in J_h\}|$. 
We have $i^1+\dots +i^h\leq \la_1+\dots+\la_h$, $l^1+\dots +l^h\leq \la_1+\dots+\la_h$
for any $h$ and $i^1+\dots +i^m=l^1+\dots +l^m=\la_1+\dots+\la_m=k$. If 
$i^h=l^h=\la_h$ for all $h$, then  $F_{i_1,\dots,i_k}, F_{l_1,\dots,l_k}$ are distinguished elements
of $\A_{J_\la}$. So we assume that at least one of the numbers $i^h,l^h$ differs from $\la_h$. 
Let $h^{max}$ be the maximal $h$ such that at least one of the numbers $i^h,l^h$ differs from $\la_h$.
Then: (a) each of $i^{h^{max}}, l^{h^{max}}$ is not less than $\la_{ h^{max}}$; (b)
 at least one of them is greater than $\la_{h^{max}}$;\  (c) $h^{max}>1$.

Then the derivative in  \Ref{derb} is zero due to the fact that the set 
$\{i_1,\dots,i_k,$ $l_1,\dots,l_k\}$ has too many elements of $\cup_{h=h^{max}}^mJ_{h}$, c.f.
formulas \Ref{fh},  \Ref{1aux}.
\end{proof}

 Theorem \ref{thmP1} is proved.
\end{proof}

Recall the elements $p_j\in \mc O(C_{\A(x),a})$, $j\in J $, and the Grothiendick residue bilinear
form $(\,,\,)_{C_{\A(x),a}}$ on $\mc O(C_{\A(x),a})$.

\begin{cor}
\label{corP1}
For any two  independent subsets $I=\{i_1,\dots,i_k\}, L=\{l_1,\dots,l_k\} \subset J$, we have
\bean
\label{P1Fo}
(p_{i_1}\dots,p_{i_k}, p_{l_1}\dots p_{l_k})_{C_{\A(x),a}}
=
(-1)^k 
 \frac{\der^{2k}P}{\der z_{i_1}\dots\der z_{i_k}\der z_{l_1}\dots\der z_{i_k}}.
\eean
\end{cor}

\begin{proof} Recall that the isomorphism of vector spaces
\bean
\label {OC=S}
\nu(x)\  : \ \on{Sing}_a\FF^k(\A(x))\ \to \ \mc O(C_{\A(x),a})
\eean
sends $\pi(F_{i_1,\dots,i_k})$ to $d_{i_1,\dots,i_k}p_{i_1}\dots p_{i_k}$ for all independent subsets $\{i_1,\dots,i_k\}\subset J$ and also identifies the
form $S^{(a)}(\,,\,)$ on $\on{Sing}_a\FF^k(\A(x))$ and the the form $(-1)^k(\,,\,)_{C_{\A(x),a}}$ on 
$\mc O(C_{\A(x),a})$, see Theorem \ref{1Mn thm}.
\end{proof}

\section {Potential of second kind}
\label{PT2}

Recall the objects  of Section \ref{secT Tr}.
Let $Q$ be the potential of second kind of the family of arrangements
$\A(x),x\in \C^n-\Delta$.

\begin{thm}
\label{thmQ1}  Let $x\in\C^n-\Delta$. Then for 
 any two  independent subsets $I=\{i_1,\dots,i_k\}, L=\{l_1,\dots,l_k\} \subset J$ and any $i_0\in J$, we have
\bean
\label{Q1F}
(p_{i_0}p_{i_1}\dots p_{i_k}, p_{l_1}\dots p_{l_k})_{C_{\A(x),a}}
=
(-1)^k
 \frac{\der^{2k+1}Q}{\der z_{i_0}\der z_{i_1}\dots\der z_{i_k}\der z_{l_1}\dots\der z_{l_k}}(x).
\eean
\end{thm}

For families of generic arrangements this theorem was proved in \cite{V6}.

\begin{proof} 
Due to relations \Ref{pREL} and \Ref{INV} it is enough to prove \Ref{Q1F} in the case when $i_0,i_1,\dots,i_k$ are distinct elements of $J$.
Thus assume that $i_0,i_1,\dots,i_k$ are distinct. After reordering  $i_0, i_1,\dots,i_k$, we may assume
that 
$i_0, i_1,\dots, i_\mu$ form a circuit, where $\mu$ is some  number $\leq k$.  Recall the function 
\bea
f_{i_0,i_1,\dots,i_k}(z)={\sum}_{j=0}^{k}(-1)^{j}z_{i_j}d_{i_0,\dots,\widehat{i_j}, \dots,i_{k}}
= {\sum}_{j=0}^{\mu}(-1)^{j}z_{i_j}d_{i_0,\dots,\widehat{i_j}, \dots, i_{\mu+1},\dots,i_{k}}
\eea
in \Ref{f_I}  and relation \Ref{propF}:
\bea
 p_{i_0}p_{i_1}\dots p_{i_{k}} = \frac 1{f_{i_0,i_1,\dots,i_k}} {\sum}_{j=0}^{\mu}(-1)^{j}
a_{i_j} w_{i_0,\dots,\widehat{i_j},\dots, i_{\mu+1}, \dots,i_{k}}.
\eea

First we analyze the left-hand side in \Ref{Q1F}. We have 
\bean
&&
\notag
(p_{i_0}p_{i_1}\dots p_{i_k}, p_{l_1}\dots p_{l_k})_{C_{\A(x),a}} 
\\
\notag
&&
= 
\frac 1{f_{i_0,i_1,\dots,i_k}} \Big( {\sum}_{j=0}^{\mu}(-1)^{j}
a_{i_j} w_{i_0,\dots,\widehat{i_j},\dots, i_{\mu+1}, \dots,i_{k}}, p_{l_1}\dots p_{l_k}\Big)_{C_{\A(x),a}}
\\
\notag
&&
=
\frac 1{d_{l_1,\dots,l_k} f_{i_0,i_1,\dots,i_k}} \Big({\sum}_{j=0}^{\mu}(-1)^{j}
a_{i_j} w_{i_0,\dots,\widehat{i_j},\dots, i_{\mu+1}, \dots,i_{k}}, w_{l_1,\dots,l_k}\Big)_{C_{\A(x),a}}
\\
\label{LP}
&&
=
\frac{(-1)^k}{d_{l_1,\dots,l_k} f_{i_0,i_1,\dots,i_k}} S^{(a)}\Big({\sum}_{j=0}^{\mu}(-1)^{j}
a_{i_j} v_{i_0,\dots,\widehat{i_j},\dots, i_{\mu+1}, \dots,i_{k}}, v_{l_1,\dots,l_k}\Big),
\eean
where the last equality holds by Theorem \ref{1Mn thm}. We have
\bean
&&
\notag
 S^{(a)}\Big({\sum}_{j=0}^{\mu}(-1)^{j}
a_{i_j} v_{i_0,\dots,\widehat{i_j},\dots, i_{\mu+1}, \dots,i_{k}}, v_{l_1,\dots,l_k}\Big)
\\
\label{LP1}
&&
=
 S^{(a)}\Big(\pi\Big({\sum}_{j=0}^{\mu}(-1)^{j}
a_{i_j} F_{i_0,\dots,\widehat{i_j},\dots, i_{\mu+1}, \dots,i_{k}}\Big), F_{l_1,\dots,l_k}\Big),
\eean
where $\pi$ is the orthogonal projection, see Section \ref{FOP}.

Let $E(i_0,i_1,\dots,i_\mu)$
 be the set of all elementary subarrangements $\A_{J_\la}(x)$ of $\A(x)$ with $J_\la=\cup_{h=1}^mJ_h$ such that
$J_1=\{i_0,i_1,\dots,i_\mu\}$ and such that 
$F_{i_1,\dots,i_{k}}$ is a distinguished element of $\A_{J_\la}(x)$.
Let $s(\A_{J_\la}(x))$ be the singular element of $\A_{J_\la}(x)$ considered as
an element of $\Sing_aV$. The singular element 
is defined up to multiplication by $\pm1$. We fix the sign so that the distinguished element $F_{i_1,\dots,i_k}$ enters
$s(\A_{J_\la}(x))$ with coefficeint 1. 

\begin{lem}
\label{lemIMP}
We have 
\bea
\phantom{aaa}
\pi\Big({\sum}_{j=0}^{\mu}(-1)^{j}
a_{i_j} F_{i_0,\dots,\widehat{i_j},\dots, i_{\mu+1}, \dots,i_{k}}\Big) =
{\sum}_{\A_{J_\la}(x) \in E(i_0,i_1,\dots,i_\mu)} 
\frac{a_J}{a(J_{\lambda},J)} s(\A_{J_{\lambda}}(x)).
\eea

\end{lem}

\begin{proof}
Indeed we have 
\bean
\label{Rhs}
\phantom{aaa}
{\sum}_{j=0}^{\mu}(-1)^{j}
a_{i_j} F_{i_0,\dots,\widehat{i_j},\dots, i_{\mu+1}, \dots,i_{k}}
=
\Big({\sum}_{j=0}^{\mu}(-1)^{j}
a_{i_j} F_{i_0,\dots,\widehat{i_j},\dots,i_\mu}\Big)\wedge F_{ i_{\mu+1}, \dots,i_{k}}.
\eean
To construct the orthogonal projection of the right-hand side in \Ref{Rhs}, we need to apply the construction of the orthogonal projection,
described in
the proof of Theorem \ref{thmORT}, but starting with step 2 since the result of the first step is already presented by the factor
\\
$\Big({\sum}_{j=0}^{\mu}(-1)^{j}
a_{i_j} F_{i_0,\dots,\widehat{i_j},\dots,i_\mu}\Big)$ in 
the right-hand side of \Ref{Rhs}, c.f. formulas \Ref{st12}-\Ref{st16}.
\end{proof}

By Lemma \ref{lemIMP} the expression in \Ref{LP} equals
\bean
\label{LP11}
\frac {(-1)^k}{d_{l_1,\dots,l_k} f_{i_0,i_1,\dots,i_k}}
{\sum}_{\A_{J_\la}(x) \in E(i_0,i_1,\dots,i_\mu)} 
\frac{a_J}{a(J_{\lambda},J)}\,
S^{(a)}( s(\A_{J_{\lambda}}(x)), F_{l_1,\dots,l_k}).
\eean
We have
\bea
s(\A_{J_{\lambda}}(x))
& =&
 \Big({\sum}_{j=0}^{\mu}(-1)^{j}
a_{i_j} F_{i_0,\dots,\widehat{i_j},\dots,i_\mu}\Big)\wedge \left({\sum}_{i=1}^{\lambda_{2}+1}
(-1)^{i+1}a_{j_{i}^{2}}F_{K_{\widehat{j_{i}^{2}}}}\right)
  \wedge\dots
\\
&&
\dots\wedge \left({\sum}_{i=1}^{\lambda_{m}+1}(-1)^{i+1}a_{j_{i}^{m}}F_{K_{\widehat{j_{i}^{m}}}}\right),
\eea
where we use the notations of Section \ref{DEl}, namely, we have
$J_{h}=\{j_{1}^{h},j_{2}^{h},\dots,j_{\lambda_{h}+1}^{h}\}$ for $h=2,\dots,m$,
and $K_{\widehat{j_{i}^{h}}}=\{j_{1}^{h},\dots,\widehat{j_{i}^{h}},\dots, j_{\lambda_{h}+1}^{h}\}$.

Due to our choice of sign of 
$s(\A_{J_{\lambda}}(x))$ we may assume that we have the equality of ordered sets
\bean
\label{7.9}
\{i_1,\dots,i_k\}=\{i_1,\dots,i_\mu, K_{\widehat{j^2_1}}, K_{\widehat{j^3_1}},\dots,K_{\widehat{j^m_1}}\}.
\eean
The term $S^{(a)}( s(\A_{J_{\lambda}}(x)), F_{l_1,\dots,l_k})$ is nonzero if and only if
\bean
\label{7.10}
\{l_1,\dots,l_k\}
 = 
\{i_0,\dots,\widehat{i_s},\dots, i_\mu, K_{\widehat{j^2_{s_2}}}, K_{\widehat{j^3_{s_3}}},
\dots,K_{\widehat{j^m_{s_m}}}\}
\eean
for some $0\leq s\leq \mu$ and some $1\leq s_h \leq \la_h+1$ for $h=2,\dots,m$. In this case
\bean
\label{7.11}
&&
\frac {(-1)^k}{d_{l_1,\dots,l_k} f_{i_0,i_1,\dots,i_k}}
\frac{a_J}{a(J_{\lambda},J)}\,
S^{(a)}( s(\A_{J_{\lambda}}(x)), F_{l_1,\dots,l_k})
\\
\notag
&&
\phantom{aaaa}
 =
\frac {(-1)^k}{d_{l_1,\dots,l_k} f_{i_0,i_1,\dots,i_k}}
\frac{a_J}{a(J_{\lambda},J)} (-1)^{s+\sum_{h=2}^ms_h} {\prod}_{q=0}^\mu a_{i_q}
{\prod}_{h=2}^m{\prod}_{q=1}^{\la_h+1} a_{j^h_q}.
\eean

Consider the  right-hand side of \Ref{Q1F}.
The potential $Q$ of second kind is the sum
\\
$\sum
\frac {a_J}{a(J_\la,J)} Q_{\A_{J_\la}}$
 shown in \Ref{P2}, where the sum is over all elementary subarrangements 
\\
$\A_{J_\la}(x)$ of $\A(x)$.  
\begin{lem}
\label{lem00}
If the derivative
\bean
\label{7.13}
 \frac{\der^{2k+1}Q_{\A_{J_\la}}}{\der z_{i_0}\der z_{i_1}\dots \der z_{i_k}\der z_{l_1}\dots
\der z_{l_k}} 
\eean
is nonzero, then $F_{i_1,\dots,i_k}$ and $F_{l_1,\dots,l_k}$ are distinguished elements of $\A_{J_\la}$.

\end{lem}
\begin{proof}
The proof is the same as the proof of Lemma \ref{lem0}.
\end{proof}

Clearly the function $\frac 1{f_{i_0,\dots,i_\mu}(z)}$ multiplied by a constant can be obtained
by this differentiation only if $\A_{J_\la}(x) \in E(i_0,i_1,\dots,i_\mu)$. In this case we have
\bean
\label{7.13}
\frac{a_J}{a(J_\la,J)} Q_{\A_{J_\la}}=
\frac{a_J}{a(J_\la,J)}
\frac {\prod_{q=0}^\mu a_{i_q}}
{(2\mu)!}
\ln (f_{\A_{J_\la,1}}) \frac{(f_{\A_{J_\la,1}})^{2\mu}}
{(\prod_{j=0}^\mu d_{i_0,\dots,\widehat{i_j},\dots,i_\mu;1})^2}
\prod_{h=2}^m P_{\A_{J_\la,h}},
\eean
see formula \Ref{2ell}.  Derivatives of this summand do not depend on the ordering the elements of the sets $J_h$, $h=1,\dots,m$, and we may assume that the equality  \Ref{7.9} of ordered sets holds.

By Lemma \ref{lem00} we may assume that the equality of ordered sets in \Ref{7.10} holds.
 In that case,
the operator $(-1)^k 
 \frac{\der^{2k+1}}{\der z_{i_0}\dots \der z_{i_k}\der z_{i_0}\dots
\widehat{\der z_{i_s}}\dots\der z_{i_\mu}\der z_{i_{\mu+1}}\dots\der z_{l_k}}$ 
applied to the expression in \Ref{7.13}
gives
\bean
\label{last}
&&
(-1)^k 
\frac{a_J}{a(J_\la,J)} 
{\prod}_{q=0}^\mu a_{i_q}  {\prod}_{h=2}^m{\prod}_{q=1}^{\la_h+1} a_{j^h_q}
\\
\notag
&&
\phantom{aaaaaaa}
\times
\frac {(-1)^{s+\sum_{h=2}^ms_h}
}{ d_{i_0,\dots,\widehat{i_s},\dots,i_\mu;1}\prod_{h=2}^m d_{j^h_1,\dots,\widehat{j^h_{s_h}},\dots, j^h_{\la_h+1} ;h}}
\frac 1{ f_{\A_{J_\la,1}}\prod_{h=2}^m d_{j^h_2,\dots,j^h_{\la_h+1};h }}.
\eean
Lemma \ref{ldec} implies that
\bean
\label{7.15}
\phantom{aaaa}
d_{i_0,\dots,\widehat{i_s},\dots,i_\mu;1}
\prod_{h=2}^m
 d_{j^h_1,\dots,\widehat{j^h_{s_h}},\dots, j^h_{\la_h+1};h }
=
d_{l_1,\dots,l_k},
\qquad
f_{\A_{J_\la,1}}\prod_{h=2}^m d_{j^h_2,\dots,j^h_{\la_h+1};h } = f_{i_0,\dots,i_k}.
\eean
Now \Ref{last} equals \Ref{7.11}. This proves Theorem \ref{thmQ1}.
\end{proof}

\begin{cor}
\label{corQ}  Let $x\in\C^n-\Delta$. Then for 
 any two ordered independent subsets $I=\{i_1,\dots,i_k\}, L=\{l_1,\dots,l_k\} \subset J$ and any $i_0\in J$, we have
\bean
\label{QRF}
\phantom{aaaa}
S^{(a)}(K_{i_0}(x) v_{i_1,\dots, i_k}, v_{l_1,\dots,l_k})
=
 d_{i_1,\dots, i_k} d_{l_1,\dots,l_k}
 \frac{\der^{2k+1}Q}{\der z_{i_0}\der z_{i_1}\dots\der z_{i_k}\der z_{l_1}\dots\der z_{l_k}}(x).
\eean
\end{cor}

\begin{proof}
The corollary follows from formula \Ref{markeD} and Theorems \ref{1Mn thm}, \ref{K/f}.
\end{proof}

\end{document}